\documentclass[11pt,reqno,a4paper, english]{amsart}
\usepackage[top=2cm,bottom=3cm,left=2.5cm,right=2.5cm]{geometry}

\usepackage{microtype}

\usepackage{pdfpages}
\usepackage[T1]{fontenc}
\usepackage[utf8]{inputenc}
\usepackage{babel}
\usepackage{amsmath}
\usepackage{amssymb}
\usepackage{amsthm}
\usepackage{textcomp}
\usepackage{enumitem}
\usepackage{dsfont}
\usepackage{mathrsfs}
\usepackage{cite}
\usepackage{color}
\usepackage[colorlinks=true,urlcolor=blue,citecolor=blue,linkcolor=blue]{hyperref}
\usepackage{mathtools, amsthm, amsfonts, amssymb}

\newcommand{\abs}[1]{\vert #1 \vert}

\newcommand{\norm}[1]{\| #1 \|}

\newcommand{\parent}[1]{\left(#1\right)}
\newcommand{\parentbig}[1]{\biggl(#1\biggr)}
\newcommand{\set}[1]{\bigl\{#1\mathclose{}\bigr\}}

\newcommand{\bigo}[1]{\mathcal{O}\mathopen{}\left(#1\right)}

\newcommand{\japbrak}[1]{\langle#1\rangle}
\newcommand{\interval}[4]{\mathopen{#1}#2\mathclose{}\mathpunct{},#3\mathclose{#4}}

\newcommand{\intervalcc}[2]{\interval{[}{#1}{#2}{]}}

\newcommand{\R}{\mathbb{R}}

\newcommand{\Z}{\mathbb{Z}}
\newcommand{\N}{\mathbb{N}}

\newcommand{\T}{\mathbb{T}}

\newcommand{\classeC}{\mathcal{C}}

\renewcommand{\epsilon}{\varepsilon}

\DeclareMathOperator{\supp}{supp}

\renewcommand\d{\,{\mathrm d}}
\newcommand\e{\,{\mathrm e}}

\newcommand{\longrightarroww}[2] {\mathop{\longrightarrow}\limits_{#1}^{#2}}

\newtheorem{mydef}{Definition}[section]
\newtheorem{thm}[mydef]{Theorem}
\newtheorem{lem}[mydef]{Lemma}
\newtheorem{prop}[mydef]{Proposition}
\newtheorem{cor}[mydef]{Corollary}

\theoremstyle{remark}
\newtheorem{rk}[mydef]{Remark}

\newtheorem*{hyp*}{Assumption}

\makeatletter
\def\subsubsection{\@startsection{subsubsection}{4}%
  \z@\z@{-\fontdimen2\font}%
  {\newline\normalfont\bfseries}}
\makeatother

\title[Pathological set for NLS]{\scshape{\textbf{Pathological set of initial data for scaling-supercritical nonlinear Schrödinger equations}}}

\author{Nicolas Camps}
\address{Universit\'e Paris-Saclay, Laboratoire de mathématiques d'Orsay, UMR 8628 du CNRS, B\^atiment 307, 91405 Orsay Cedex,}
\email{nicolas.camps@universite-paris-saclay.fr}

\author{Louise Gassot}
\address{ICERM, Brown University, 121 South Main Street, Providence, RI 02903, USA
}
\email{louise.gassot@normalesup.org}



\date{\today}

\begin{document}

\begin{abstract}

The purpose of this work is to evidence a pathological set of initial data for which the regularized solutions by convolution experience a norm-inflation mechanism, in arbitrarily short time. The result is in the spirit of the construction from Sun and Tzvetkov~\cite{SunTzvetkov2020pathological}, where the pathological set contains superposition of profiles that concentrate at different points. Thanks to finite propagation speed of the wave equation, and given a certain time, at most one profile exhibits significant growth. However, for Schrödinger-type equations, we cannot preclude the profiles from interacting between each other. Instead, we propose a method that exploits the regularizing effect of the approximate identity which, at a given scale, rules out the norm inflation of the profiles that are concentrated at smaller scales.
\end{abstract}

\ \vskip -1cm  \hrule \vskip 1cm \vspace{-0.5cm}
 \maketitle 
{ \textwidth=4cm \hrule}

\maketitle

\tableofcontents
\vspace{-1.5cm}
\section{Introduction}

Our model will be the semi-linear Schrödinger equation on $\R^3$ 
\begin{equation}\label{eq:NLS}
\tag{NLS}
i\partial_t u+\Delta u+\sigma|u|^{p-1}u=0\,,\quad (t,x)\in\R\times\R^3\,,\quad p>1\,,
\end{equation}
where $\sigma=1$ in the focusing case and $\sigma=-1$ in the defocusing case, respectively. Our main result is the description of a {\it pathological set} for equation~\eqref{eq:NLS}, based on the approach of Sun and Tzvetkov~\cite{SunTzvetkov2020pathological} for the nonlinear wave equation. We prove that if we regularize the initial data using a given convolution profile, there exists a dense set of initial data in~$H^s(\R^3)$ such that the regularized solutions exhibit norm inflation.  The precise statement can be found in Theorem~\ref{thm-NLS} below. Such a result aims at precising the probabilistic well-posedness theory, summarized in our context in Theorem~\ref{theo-lwp}. Let us first recall the motivation for this result. In the study of the Cauchy theory for dispersive equations, one generally encounters a critical regularity threshold that dictates whether the equation is well-posed or not, in the Hadamard sense. We recall that a general dispersive equation on a Riemannian manifold $(M,g)$ can be written under the form
\begin{equation}
\label{eq:disp}
i\partial_t u + Lu = F(u)\,,\quad u(0)\in H^s(M)\,,  
\end{equation}
where $L$ is a well-defined operator on $H^s(M)$ such that the linear propagator  $\set{\exp(itL)}_{t\in\R}$ is a group of isometries of $H^s(M)$, whereas the nonlinear interaction, denoted $F(u)$, has to make sense in $H^s(M)$. Given a time interval $I\subset \R$ containing $t_0$, the curve $t\in I\subset\R\mapsto u(t)\in C(\intervalcc{t_0}{T},H^s(M))$ is a strong solution to~\eqref{eq:disp}, with initial condition $u(t_0)=u_0\in H^s(M)$, if it satisfies the following Duhamel integral representation formula, in the distributional sense
\begin{equation}
    \label{eq:duhamel}
    u(t) = \e^{itL}u_0 -i\int_{t_0}^t\e^{i(t-\tau)L}F(u(\tau))d\tau\,.
\end{equation}
For dispersive equations, a suitable notion of well-posedness, that we recall here, can be found in~\cite{TzvetkovIll2007}.
\begin{mydef}[Local well-posedness]
\label{def:wp}
Given a general dispersive equation as~\eqref{eq:disp}, we say that the Cauchy problem is well-posed in $H^s(M)$ if for every bounded set $B\subset H^s(M)$, there exist $T=T(B)>0$ and a Banach space $X_T$ continuously embedded in $C(\intervalcc{t_0}{T}, H^s(M))$ such that for any $u_0\in H^s(M)$, there exists a solution $u\in X_T$ satisfying the Duhamel integral representation formula~\eqref{eq:duhamel}. Moreover, the following holds.
\begin{enumerate}
    \item (Uniqueness)~\footnote{This statement is sometimes referred as \emph{conditional uniqueness}, whereas unconditional uniqueness in $H^s(M)$ holds when it is possible to prove uniqueness in $C(\intervalcc{t_0}{T},H^s)$ and not just in $X_T$.} Let $u_1, u_2\in X_T$ satisfying~\eqref{eq:duhamel}. If, for some $t_0\leq t\leq T$, $u_1(t)=u_2(t)$, then $u_1=u_2$. 
    \item (Continuity) The flow map $u_0\in B\mapsto u\in X_T$ is continuous.
    \item (Persistence of regularity) If $u_0\in H^\sigma(M)$ with $\sigma>s$, then $u\in C(\intervalcc{t_0}{T}),H^\sigma(M))$.
\end{enumerate}
In addition, when the flow-map is uniformly continuous, we say that the Cauchy problem is semi-linearly well-posed.
\end{mydef}
This definition deserves some comments, especially in the case of critical regimes where the time of existence does not only depend on the bounded set $B$ but also on the profile of the initial data. In Definition~\ref{def:wp}, we ask for some persistence of regularity in order to see the solutions in $H^s(M)$ as limits of regularized smooth solutions, and this motivates the study of the Cauchy problem at low regularity. A natural question that arises is whether the flow map is always uniformly continuous, as soon as it is continuous. In other words, is the critical threshold for the well-posedness equals to the critical threshold for the semi-linear well-posedness ? In general, the answer is no, especially in the context of KdV type equations, see~\cite{Bourgain1993KdV, TakaokaTsutsumi2004}.  We refer to~\cite{TzvetkovIll2007} for discussions on this  question, and for further comments on Definition~\ref{def:wp}.  

Sometimes, an invariance property of the equation under the action of a group of symmetries indicates the aforementioned well-posedness threshold. The first example, that we exploit in this work, is the \textit{scaling} symmetry. If an equation is invariant under a scaling symmetry
\[
u(t,x) \longmapsto u_\lambda(t,x) :=\lambda^{\alpha} u(\lambda^\beta t,\lambda^\gamma x)\,,\quad \lambda>0\,,
\]
then, in the scale of homogeneous Sobolev spaces, its critical regularity $s_c$ corresponds to the Sobolev norm left invariant under the action of the scaling 
\[
\norm{u_\lambda(0,x)}_{\dot{H}^{s_c}} = \norm{u(0,\cdot)}_{\dot{H}^{s_c}}\,.
\]
We expect well-posedness at \emph{scaling-subcritical} regularities $s\geq s_c$, and ill-posedness at \emph{scaling-supercritical} regularities $s<s_c$. Note, however, that this heuristic is quite informal and in general, there might exist other properties of the equation that actually dictate the well-posedness threshold.

For instance, in the case of integrable equations, the scaling, semi-linear and well-posedness thresholds often differ. A classic example is the KdV equation on $\T$, which is continuously well-posed in $H^s(\T)$ if $s\geq -1$, but the solution map is locally uniformly continuous only when $s\geq -\frac{1}{2}$ (in which case it is real analytic). Besides, the next equation in the KdV hierarchy is continuously well-posed for $s\geq -\frac{1}{2}$, but the solution map is nowhere locally uniformly continuous for this range of exponents~\cite{KappelerMolnar2018kdv2}, whereas these equations are invariant under the scaling transformation with critical Sobolev index $s_c=-\frac{3}{2}$. Similarly, the Benjamin-Ono equation on $\T$ is continuously well-posed in $H^s(\T)$ for $s>-\frac{1}{2}$~\cite{GerardKappelerTopalov2020}, the flow map is locally uniformly continuous only when  $s\geq 0$ (in which case it is real analytic)~\cite{GerardKappelerTopalov2021analyticity},. The next equation in the Benjamin-Ono hierarchy is continuously well-posed for $s\geq 0$~\cite{bo_hierarchie}, but an adaptation of~\cite{KappelerMolnar2018kdv2} implies that the flow map is nowhere uniformly continuous. For all the equations in the Benjamin-Ono hierarchy, the scaling-critical Sobolev exponent is $s_c=-\frac{1}{2}$. Geometry also influences the dispersion property of the equation, for instance the NLS equation on $\T$, with scaling-critical Sobolev exponent $s_c=-\frac{1}{2}$, is ill-posed for $s<0$~\cite{ChristCollianderTao2003instability} but continuously well-posed when $s\geq 0$ with a smooth solution map~\cite{Bourgain1993}. In dimension $2$, NLS is well-posed on $\T^2$ for $s_c=0<s$, whereas the flow map of NLS on the sphere $\mathbb{S}^2$ is not uniformly locally well-posed as soon as $s<\frac{1}{4}$, see~\cite{Banica04Sphere,BGT2002Insta}. We also point out that the flow map for the NLS equation on the Heisenberg group cannot be $\classeC^3$ in $H^s(\mathbb{H}^1)$ for $s<2$, whereas the scaling-critical exponent is $s_c=1$, see for instance~\cite{GerardGrellier2010}.
 
For the nonlinear Schrödinger equation~\eqref{eq:NLS}, as well as for the nonlinear wave equation in $\R^3$ or $\T^3$, the scaling symmetry 
\[
u\longmapsto u_{\lambda}(t,x)=\lambda^{\frac{2}{p-1}} u(\lambda^2 t,\lambda x)
\]
actually dictates the local well-posedness threshold, which is independent of the focusing or defocusing nature of the equation, see~\cite{ChristCollianderTao2003norm}.

\begin{thm}[Well-posedness~\cite{CazenaveWeissler1990,ChristCollianderTao2003norm}]
Let $p>1$.~\footnote{If $p$ is not an odd integer we further assume that $\lfloor s \rfloor + 1 < p$, so that the nonlinearity makes sense in $H^s$. }
\begin{itemize}
    \item The Cauchy problem associated to~\eqref{eq:NLS} is well-posed  in $H^s(\R^3)$ in the sense of Definition~\ref{def:wp} when $\max(0,s_c)<s$.
    \item When $0\leq s_c$, the Cauchy problem is well-posed in $H^{s_c}(\R^3)$ but the time of existence does depend on the initial data and  not just on its $H^{s_c}$-norm.
    \item When $s<s_c$ or when $s<0$, the Cauchy problem is ill-posed. 
\end{itemize}
\end{thm}
We comment on the ill-posedness result due to~\cite{ChristCollianderTao2003norm} in the next paragraph. Before that, we recall that the mass and the energy are formally conserved under the flow, we denote them

\[
M(u) = \frac{1}{2}\int_{\R^3} \abs{u(x)}^2\d x\,,\quad H(u) = \frac{1}{2}\int_{\R^3} \abs{\nabla u(x)}^2\d x - \frac{\sigma}{p+1}\int_{\R^3}\abs{u(x)}^{p+1} \d x\,.
\]
Note that these conservation laws ensure the existence of smooth global solutions in the energy-subcritical regime $p\leq 5$, when the energy is coercive.

\subsubsection{Ill-posedness issues.}  Assume that the Cauchy problem is known to be well-posed for smooth initial data in the sense of Definition~\ref{def:wp}, say in $H^\sigma$ for some $\sigma$ large enough. Then, one can try to prove ill-posedness in a low-regularity space $H^s$ by showing that the flow-map, defined on $H^\sigma$ for some $\sigma>s$, does not extend continuously to $C(\intervalcc{0}{T},H^s)$, with $T$ being arbitrarily small. This strategy is motivated by the persistence of regularity of the flow, that enables one to see the solutions in $H^s$, if they exist, as a limit of smoother solutions in $H^\sigma$. 

Before going further, let us point out that it is easier to prove ill-posedness in the presence of special nonlinear solutions. For instance, for focusing Schrödinger equations where nonlinear solutions are accessible, we have the following~\cite{SulemSulem1999}. 
 If $s_c<0$, the Schrödinger equation is mass-subcritical and there exist blow-up solutions that, combined with scaling arguments and virial identities, lead to solutions blowing up in arbitrarily short time in $H^s$ when $s<s_c$. In the mass-critical case, a blow-up solution is obtained by applying a pseudo-conformal transformation to the ground state solution, and for the mass-supercritical case~\cite{KenigPonceVegaIll} proved that the flow map is not uniformly continuous, respectively. Meanwhile, in the defocusing case, we do not have explicit nonlinear profiles generating blow-up solutions, except in the Schrödinger case with $p=3, d=1$ where one can make use of modified scattering solutions~\cite{ChristCollianderTao2003Freq}. The norm inflation mechanism we describe in the next paragraph is by no mean the only way to indicate that a dispersive PDE is ill-posed, and we refer to~\cite{TzvetkovIll2007,KenigPonceVegaIll} for the description of ill-posedness results in a broader context.

\subsubsection{Norm inflation} In the case of the Schrödinger or wave type equations, the discontinuity of the flow-map at scaling-supercritical regularities is stronger. Indeed, a norm inflation mechanism occurs, and consists in constructing a sequence of smooth initial data going to zero in $H^s$, such that the norm of the solution in $H^s$ goes to infinity for arbitrarily short times. Norm inflation has been evidenced by~\cite{Lebeau2001norm,Lebeau2005norm,Lindblad1993} for the wave equation~\footnote{In~\cite{Lindblad1993}, the norm inflation is due to another concentration phenomenon related with the Lorentz invariance of the wave equation.}, and extended to the Schrödinger equation in~\cite{ChristCollianderTao2003norm}. 

The general strategy to prove norm inflation in the scaling-supercritical regime is to perform a \emph{small dispersion analysis}. Namely, by rescaling an arbitrary bump function, one generates a sequence of smooth initial data $(\psi_n)_n$ that go to zero in $H^s$, while spatially concentrating around a point. Then, one considers the smooth solution $u_n$ with initial data $\psi_n$, and studies the \emph{bubble} solution $v_n$ to the dispersionless ODE
\begin{equation*}
    \begin{cases}
    i\partial_t v_n = \sigma \abs{v_n}^{p-1}v_n\,,\\
    v_n(0) = \psi_n\,.
    \end{cases}
\end{equation*}
The ODE profile $v_n$, that captures the oscillations of $u_n$, grows in $H^s$ after a time $t_n$ going to zero as $n$ goes to infinity. Then, one uses a priori energy estimates up to time $t_n$ and establishes that uniformly in $n$,
\[
\norm{u_n(t_n)-v_n(t_n)}_{H^s}\lesssim1\,.
\]
Therefore, when $0<s_c$ and $s<s_c$, the oscillations are stronger than the dispersion. In some sense, the norm inflation mechanism reflects a frequency cascade from low Fourier modes to high Fourier modes, responsible for the discontinuity of the flow-map.

 We note that such a mechanism is independent of the focusing or defocusing nature of the equation. Specifically, in the case of Schrödinger equation, the results from~\cite{ChristCollianderTao2003norm} state as follows. Norm inflation occurs when $s<-\frac{d}{2}$, and also in the mass-supercritical case when $0<s<s_c$. Furthermore, by using the frequency modulation method, Christ Colliander and Tao evidenced in~\cite{ChristCollianderTao2003norm} the \emph{high to low} frequency cascade when $s<-\frac{d}{2}$, and the \emph{low to high} frequency cascade when $0<s$. They also proved that the flow map fails to be uniformly continuous for $s<0$ by using the Galilean symmetries of~\eqref{eq:NLS}. Note that a similar small dispersion analysis has been performed earlier, in~\cite{Kuksin1987}. We refer to~\cite{burq-tzvetkov-2008I,BurqGerardTzvetkov2005multilinear} for some subsequent norm inflation results on compact manifolds, with similar proofs that rely on concentrating profiles in a point. We also mention the norm inflation results in negative Sobolev spaces from~\cite{Kishimoto,OhInflation}. In the first one, norm inflation occurs at zero for polynomial Schrödinger equations. In the second one, norm inflation occurs at any data for the periodic cubic Schrödinger equation, respectively.  

\subsubsection{Some variations} The norm inflation does not only occur around the zero solution, but also around any solution, as shown in~\cite{Xia2021generic} for the wave equation. The proof does not rely on finite propagation speed, so that it can be adapted to~\eqref{eq:NLS}, see also~\cite{Xia2022} establishing the norm inflation at any point for the fourth-order Schrödinger equation. The general idea is that if the regularizing sequence is arbitrary, norm inflation is very likely to occur for a dense set of initial data. However, as presented in the next paragraph, the picture is more intricate when one uses the convolution by an approximate identity, which is the most natural regularization procedure. It has the property to commute with the linear flow and to be uniformly bounded on $L^p$ when $p<\infty$.  

To push this result further, a key construction consists in gluing together blow-up profiles that concentrate at different points. Such a construction is inspired from the work of Lebeau~\cite{Lebeau2001norm}, and was used in~\cite{burq-tzvetkov-2008I,Xia2021generic} for wave equations to generate initial data from which any \emph{solution} $u$ satisfying the finite propagation speed inflates instantaneously
\[
\underset{t\to 0^+}{\limsup}\norm{u(t)}_{H^s} = +\infty\,.
\]
To give a precise meaning of the above inflating solution associated with such a pathological initial data,~\cite{SunTzvetkov2020pathological} adapted the construction in the presence of a regularizing convolution. As described in the next paragraph, this is also motivated by the probabilistic Cauchy theory at supercritical regularities.

Let us conclude this paragraph by mentioning the result of~\cite{Lebeau2001norm}, that evidenced a loss of regularity mechanism for the energy-supercritical wave equation by using some nonlinear geometric optics. For Schrödinger equation we refer to~\cite{AlazardCarles2009loss,ChironRousset2009}, that indicate as well loss of regularity. 

\subsubsection{Pathological set in the context of probabilistic well-posedness}

After the pioneering contributions from Bourgain~\cite{bourgain94,bourgain96-gibbs} for the cubic Schrödinger equation on $\mathbb{T}^2$, followed by the works from Burq and Tzvetkov for the nonlinear wave equation on compact Riemannian manifolds without boundary~\cite{burq-tzvetkov-2008I,burq-tzvetkov-2008II}, many authors have investigated the well-posedness issue at scaling-supercritical regularities by using a statistical approach. Specifically, given a Sobolev space $H^s(M)$ where the equation is known to be ill-posed, or semi-linearly ill-posed, one can search for a non-degenerate probability measure supported on a dense subset of $H^s(M)$, and prove that the equation is well-posed for initial data on the support of this measure. Such an approach stems from the study of Gibbs measures~\cite{bourgain96-gibbs,BurqThomannTzvetkov-Gibbs}, and has been extensively developed in other contexts since then. We refer to~\cite{BenyiOhPocovnicuSurvey2019} for a comprehensive survey on the so-called probabilistic Cauchy theory. Concerning the cubic Schrödinger equation in the Euclidean space $\R^d$, with $d\geq3$, Bényi, Oh, and Pocovnicu~\cite{benyi2015-local} proved generic well-posedness by using a Wiener randomization procedure, based on a unit-scale decomposition of the frequency space that leads to refined Strichartz estimates. To state the result in the context of the present work, we fix an approximate identity $(\rho_{\varepsilon})_{\varepsilon>0}$ defined as
\begin{equation}
\label{eq:rho-eps}
    \rho_{\varepsilon}(x):=\frac{1}{\varepsilon^3}\rho\left(\frac{x}{\varepsilon}\right)\,,
\end{equation}
where $\rho\in\classeC_c^{\infty}(\R^3)$, valued in $[0,1]$, satisfies $\int_{\R^3}\rho(x)\d x=1$. Let us consider the cubic~\eqref{eq:NLS} equation, with $p=3$, at a scaling-supercritical regularity $s< s_c=\frac{1}{2}$.
\begin{thm}[Local well-posedness on a full-measure set for cubic~\eqref{eq:NLS}]
\label{theo-lwp}
Let $\frac{1}{4}<s<\frac{1}{2}$. There exist a probability measure $\mu$ supported on $H^s(\R^3)$ and a dense set $\Sigma\subset H^s(\R^3)$ with full $\mu$-measure, such that the following holds. For all $f_0$ in $\Sigma$, the smooth local solutions $(u^\epsilon)$ to the cubic~\eqref{eq:NLS} with initial data $(f_0\ast\rho_\epsilon)$ are well-defined up to a time $T=T(f_0)>0$, and they converge to a limiting object $u$
\begin{equation}
    \label{eq:approx-proba}
    \underset{\epsilon\to0}{\lim} \norm{u^\epsilon-u}_{L^\infty(\intervalcc{0}{T};H^s(\R^3))} =0\,,\quad u \in C(\intervalcc{0}{T}; H^s(\R^3))\,.
\end{equation}
Moreover, $u$ solves~\eqref{eq:NLS} in the distributional sense~\eqref{eq:duhamel} on $\intervalcc{0}{T}$, with initial data $f_0$.
\end{thm}

The local well-posedness result is due to~\cite{benyi2015-local} in dimension 4, and~\cite{benyi2015} in dimension $d\geq3$ for cubic~\eqref{eq:NLS}, written in the functional framework of $U^p, V^p$ spaces  in order to prove scattering for small data. However, the convergence of the approximate solutions~\eqref{eq:approx-proba} is not written in~\cite{benyi2015}, and we prove it in Section~\ref{sec:proba} by adapting the analysis performed in~\cite{benyi2015-local,benyi2015}. Note that the regularization by convolution is essential since, as pointed out in~\cite{Xia2021generic} for the wave equation, norm inflation occurs over a dense set of initial data when the regularization procedure is arbitrary. This result is in the spirit of Theorems~2.6 and~2.7 form~\cite{TzvetkovRandomWave} for the wave equation. Let us mention that the large-time behavior of the probabilistic solution has been investigated by many authors, by combining the probabilistic local well-posedness result of Theorem~\ref{theo-lwp} with deterministic methods. Concerning the cubic Schrödinger equation in $\R^3$,~\cite{benyi2015,PoiretRobertThomann2014} established almost-sure scattering for small data, and~\cite{camps2020} proved almost-sure asymptotic stability for small ground states. Outside the small data regime,~\cite{killip-murphy-visan-2019,dodson-luhrmann-mendelson-20,Spitz2021} proved almost-sure scattering for the cubic in $\R^4$, and ~\cite{camps2021,SofferShenWu2021} extended the scattering result in $\R^3$, where the cubic~\eqref{eq:NLS} is energy-subcritical. 

Outside this full-measure set where the statistical approach provides a strong notion of well-posedness, there exist initial data that concentrate in a point so that a low to high frequency cascade occurs. From the result of~\cite{Xia2021generic}, we know that if the regularization procedure is arbitrary, such an instability phenomenon may appear in the vicinity of any initial data. In addition, Sun and Tzvetkov proved in~\cite{SunTzvetkov2020pathological} that for nonlinear wave equations, when the initial data are regularized by convolution, norm inflation occurs on a {\it pathological set} containing a dense $G_\delta$ set, even when a probabilistic well-posedness result in the spirit of Theorem~\ref{theo-lwp} holds. 

\subsubsection{Main result}
Following the approach of Sun and Tzvetkov~\cite{SunTzvetkov2020pathological} for the nonlinear wave equation, we describe the pathological set for equation~\eqref{eq:NLS}, which is the counterpart of the probabilistic well-posedness theory presented in Theorem~\ref{theo-lwp}. We show that even if we naturally regularize the initial by convolution, there still exist norm inflation solutions arising from a dense set of initial data in~$H^s$.
\begin{thm}[Pathological set of scaling-supercritical data for~\eqref{eq:NLS}]
\label{thm-NLS}
Let $p\geq3$ be an odd integer~\footnote{This assumption is technical, and might be removed. It guarantees that the bubbles constructed in paragraph~\ref{sec:bubble} are $C^\infty$.} and $0< s<s_c=\frac{3}{2}-\frac{2}{p-1}$. There exists a dense set $S\subset H^s(\R^3)$ such that for every $f_0\in S$, the family of local solutions $(u^{\varepsilon})$ of~\eqref{eq:NLS}  with initial data $(f_0\ast\rho_{\varepsilon})$ does not converge because of norm inflation. More precisely, there exist $\varepsilon_n\to 0$ and $t_n\to 0$ such that $u^{\varepsilon_n}$ exists in $C(\intervalcc{0}{t_n},H^s(\R^3))$, and
\begin{equation}
    \label{eq:limsup}
    \lim_{n\to\infty}\|u^{\varepsilon_n}(t_n)\|_{H^s(\R^3)}=\infty\,.
\end{equation}
Moreover, in the defocusing energy-subcritical and critical cases $p\in\set{3,5}$, where~\eqref{eq:NLS} is known to be globally~\footnote{We refer to~\cite{CKSTTCritical2008} for the proof of global well-posedness and scattering in the energy space of the quintic ($p=5$) defocusing Schrödinger equation.} well-posed in $H^1$, the pathological set $\mathcal{P}$ contains a dense $G_\delta$ set.
\end{thm}

To exhibit the pathological set, we use the same construction as in~\cite{SunTzvetkov2020pathological} that consists in putting side by side an infinity of inflating bubbles concentrated at arbitrarily small scales. As discussed above, this construction is contained in~\cite{burq-tzvetkov-2008I,ChristCollianderTao2003norm}, and was inspired from~\cite{Lebeau2001norm}. In order to generate a dense set in $H^s$ and not only around zero,~\cite{SunTzvetkov2020pathological} implemented the idea from~\cite{Xia2021generic} that consists in performing the norm inflation around any data in $H^s$. The proof proposed in~\cite{SunTzvetkov2020pathological} revisits the small dispersion analysis from~\cite{ChristCollianderTao2003norm} with an additional convolution. Indeed, since their motivation is to make precise the probabilistic well-posedness result, the smooth initial data need to be obtained by convolution with an approximate identity. Then, they used finite propagation speed to indicate that the different bubbles do not interact in short time.

In the case of Schrödinger type equations, we do not have such a finite propagation speed principle. To overcome this, we play with the regularizing convolutions. The observation is that given a fixed scale of regularization $\epsilon>0$, the convolution by $\rho_\epsilon$ breaks the bubbles that concentrate at scale $\lambda^{-1}\ll\epsilon$. Hence, at a certain suitable time $t_\epsilon\ll1$, the bubble concentrated at scale $\lambda^{-1}\sim\epsilon$ is the only one to inflate. This observation indicates that there is no need to use the finite propagation speed in the proof of Theorem~\ref{thm-NLS}, even in the case of the nonlinear wave equation analyzed in~\cite{SunTzvetkov2020pathological}. Therefore, the same approach probably extends to any equation of Schrödinger type, provided there exists a reasonable Cauchy theory for smooth initial data. Indeed, we do not directly use the particular dispersion of~\eqref{eq:NLS} in the analysis. Let us mention features of the approach developed in this work.
\subsubsection*{\emph{\textmd{Finite propagation speed and localization of bubbles}}}
Since we do not use any finite propagation speed, we have access to Schrödinger type equations. Another consequence is that we can superpose the bubbles in the construction. Indeed, we do not exploit the absence of interaction between the different bubbles, that was crucial in the short time analysis from~\cite{SunTzvetkov2020pathological} about the wave equation. Hence, we can take the center of the bubbles to be $x_k = 0$, so that the pathological set is left invariant under the action of the rotations and the construction persists in the radial case. Note, however, that we are not able to prove that the solution inflates along the whole sequence as $\epsilon$  goes to zero, but only along the subsequence $\epsilon_n$ (see remark~1.1 in~\cite{SunTzvetkov2020pathological}, and remark~\ref{rem:eps} below). 

\subsubsection*{\emph{\textmd{Focusing and defocusing equations and range of nonlinearities}}}
We show that the a priori control on the growth of Sobolev norms is sufficient to establish a lower bound on the time of existence for the smooth solutions obtained by convolution. Hence, we can bypass the use of the coercive energy, and we treat the focusing and defocusing case without distinction. Also, we have access to the energy-supercritical case $p>5$. In such a case, we point out that the norm inflation of the $H^1(\R^3)$-norm does not contradict the conservation of the energy, since when $p>5$ the potential energy of an initial data in $H^1(\R^3)$ is infinite.  However, to prove that $\mathcal{P}$ contains a dense $G_\delta$ set we need to assume that the smooth solutions are global, say in $H^{\frac{3}{2}+}(\R^3)$, see Proposition~\ref{prop:dense}. Hence, the reason for the technical constraint on $p\in\set{3,5}$ in the statement of the second part of Theorem~\ref{thm-NLS}.

\subsubsection*{\emph{\textmd{Lower bound on \texorpdfstring{$s$}{s}}}} Sun and Tzvetkov observed in~\cite{SunTzvetkov2020pathological,SunTzvetkov2020JFA} that a lower bound on $s$ might be required. However, in the statement of Theorem~\ref{thm-NLS}, we do not need such a lower bound, and we propose an explanation of this difference in remark~\ref{rem:lower-bound}. Yet, we do not know if the limitation in the case of the wave equation is just technical or not. Note that we do not have $s=0$ because of mass conservation. When $s<0$ other mechanisms operate, and the possible norm inflation is due to a \emph{high to low} frequency cascade. 

\subsubsection*{\emph{\textmd{Other geometries and dimensions}}}
Since the construction of the pathological set is local, up to working in local coordinates, our proof works as well in the periodic case, or in the case of a compact Riemannian manifold. Nevertheless, we decided to state the result in the Euclidean case, since there is no general probabilistic Cauchy theory in the spirit of~\cite{burq-tzvetkov-2008I} for~\eqref{eq:NLS} on a compact manifold of dimension 3. Our proof can also be adapted to higher dimension, but we write it in the context of~\cite{SunTzvetkov2020pathological} for simplicity, in the sense that we only need to compute two derivatives to control the $L^\infty$-norm of a function in $\R^3$. In higher dimensions, we would have to use a semiclassical energy with higher order derivatives. 

\subsubsection*{Conclusion}

In the scaling-supercritical regime, and at least in the cubic case, we have two different behaviors for the smooth solutions to~\eqref{eq:NLS} initiated from regularized initial data obtained by convolution with an approximate identity. On the one hand, Theorem~\ref{thm-NLS} states that the {\it pathological} set of initial data where norm inflation occurs contains a dense $G_\delta$ set. On the other hand, the statistical approach from Theorem~\ref{theo-lwp} provides of dense full-measure set of initial data for which the regularized solutions converge to a limiting object that is a strong solution to~\eqref{eq:NLS}. This is the counterpart for Schrödinger-type equations to the result from Sun and Tzvetkov~\cite{SunTzvetkov2020pathological} for wave equations, and the same comments apply. Namely, there are two different notions of genericity. While the norm inflation set from Theorem~\ref{thm-NLS} is generic for the topology, it is negligible for the measure. Meanwhile, we know from the Baire category Theorem that the well-posedness set from Theorem~\ref{theo-lwp}, which has full measure, cannot be a dense $G_\delta$ set.

\subsubsection{Outline of the paper}

We first establish the generic ill-posedness Theorem~\ref{thm-NLS} in section~\ref{section:pathological}. To achieve such a goal, we set up the fundamental norm inflation bubbles $v_n$ in part~\ref{sec:bubble},  and we construct the pathological set in part~\ref{sec:pathological} by superposing these bubbles. In paragraph~\ref{sec:perturbative-analysis} we fix a given scale $\epsilon_n$, its corresponding inflation time $t_n$, and we compare in $H^s(\R^3)$ the solution to~\eqref{eq:NLS} initiated from an initial data in the pathological set, with the bubble concentrated at scale $\epsilon_n$, supposed to inflate after the time $t_n$. In a second section~\ref{sec:proba}, we prove the generic well-posedness result. After introducing preliminaries for the probabilistic method in part~\ref{sec:prelim} and reviewing relevant probabilistic Cauchy theory results in part~\ref{sec:proba-lwp}, we prove in part~\ref{sec:proba-convol} the convergence in $H^s(\R^3)$ of the regularized solution by convolution to the solution generated from a randomized initial data.

\begin{center}
\textbf{Acknowledgements}
\end{center}
The authors would like to thank Chenmin Sun and Nikolay Tzvetkov for presenting them the problem and the \emph{``tanghuru''} construction. They also thank Nicolas Burq and Frédéric Rousset for helpful discussions, and Tristan Robert for some references.

\section{Pathological set of data where norm inflation occurs}\label{section:pathological}

The purpose of this section is to prove Theorem~\ref{thm-NLS}. First, we follow the strategy from~\cite{SunTzvetkov2020pathological} and we study the unstable profile, or the \emph{bubble}, regularized by an approximate identity. Our key observation can be stated as follows. When convoluted with an approximated identity $(\rho_\epsilon)$, a blow-up profile concentrated at scale $n^{-1}$, and supposed to inflate at time $t_n$, actually gets smoother and remains bounded at time $t_n$, as soon as $n^{-1}\ll \epsilon$. We draw the attention to Lemmas~\ref{lem:growth_vn} and~\ref{lem:scale} that indicate such a mechanism. Owing to this observation, we show in Section~\ref{sec:perturbative-analysis} how to generalize the small-dispersion analysis from~\cite{SunTzvetkov2020pathological} without using finite propagation speed. 

\subsection{Unstable profile}\label{sec:bubble}

We construct unstable profiles by following the approach in~\cite{BurqGerardTzvetkov2005multilinear}. The construction of the bubbles, that have growing $H^s$-norm, is based on the solution $V(t)=e^{i\sigma t}$ to the dispersionless ODE  
\[
\begin{cases}
iV'+\sigma |V|^{p-1}V=0
\\
V(0)=1\,.
\end{cases}
\]
Note that $V$ is a bounded periodic function. Set a smooth and radial profile  $\varphi\in\classeC_{c}^{\infty}(\{\|x\|\leq 1\})$, such that $0\leq \varphi\leq 1$, from which we define a bubble concentrated at scale $n$ 
\begin{equation*}
v_n(0,x):=\lambda_n^{\frac{2}{p-1}}\varphi(nx)=\kappa_n n^{\frac{3}{2}-s}\varphi(nx)\,.
\end{equation*}
Given $\gamma>0$ to be determined later in the analysis, the parameters of the above bubble are
\[
\kappa_n:=(\log(n))^{-\gamma}
,\quad \lambda_n:=\kappa_n^\frac{p-1}{2} n^{(\frac{3}{2}-s)\cdot\frac{p-1}{2}}\,.
\]
We recall that we have fixed an approximate identity defined in~\eqref{eq:rho-eps},
and consider the bubble after convolution by the approximate identity
\[
v_n^{\varepsilon}(0):=\rho_{\varepsilon}\ast v_n(0).
\]
Given a time $t$, we consider the profile
\begin{equation*}
v_n^{\varepsilon}(t,x):=v_n^{\varepsilon}(0,x) V(t|v_n^{\varepsilon}(0,x)|^{p-1})\,,
\end{equation*}
solution to the dispersionless ODE
\begin{equation*}
    \begin{cases}
        i\partial_t v_n^{\varepsilon}+\sigma|v_n^{\varepsilon}|^{p-1}v_n^{\varepsilon}=0\,.\\
        v_n^{\varepsilon}(0) = \rho_{\varepsilon}\ast v_n(0)\,,
    \end{cases}
\end{equation*}
We now fix the parameters,
\begin{equation}
    \label{eq:parameters}
    \varepsilon_n:=\frac{1}{100n},\quad
t_n:=\log(n)^{\beta(p-1)}n^{(s-\frac{3}{2})(p-1)}=\lambda_n^{-2} \log(n)^{(\beta-\gamma)(p-1)}\,,
\end{equation}
for suitable $0<\gamma<\beta<1$ adjusted in the analysis, and we follow~\cite{SunTzvetkov2020pathological} to adapt the bounds from~\cite{BurqGerardTzvetkov2005multilinear} in the presence of the convolution parameter.
\begin{lem}[Growth of the profile]\label{lem:growth_vn} Let $0\leq s<\frac{3}{2}$.~\footnote{Note that when $s=0$ the lower bound goes to zero as $n$ goes to zero. This is in agreement with the mass conservation.}
\begin{enumerate}
\item There exists $c>0$ independent of $n$ such that for all $\varepsilon\leq\varepsilon_n$, we have the lower bound 
\[
\|v_n^{\varepsilon}(t_n)\|_{H^s}\geq c \kappa_n(\lambda_n^2t_n)^s\,.
\]
\item For all $m\in\N$, there exists $C>0$ such that for all $n\in\N$, $t\in\R$ and $\epsilon>0$,
\[\||\nabla|^m v_n^{\varepsilon}(t)\|_{L^2}\leq C \kappa_n n^{m-s}(1+(t\lambda_n^2)^{m})\min\left(1,\frac{1}{\varepsilon n}\right)^{m+\frac{3}{2}}\,.
\]
In addition, 
\[\||\nabla|^m v_n^{\varepsilon}(t)\|_{L^{\infty}}\leq C \kappa_n n^{m+\frac{3}{2}-s}(1+(t\lambda_n^2)^{m})\min\left(1,\frac{1}{\varepsilon n}\right)^{m+\frac{3}{2}}\,.
\]
\end{enumerate}
\end{lem}

The factor $\min\left(1,\frac{1}{\varepsilon n}\right)^{m+\frac{3}{2}}$ obtained on top of Lemma 2.1 in~\cite{SunTzvetkov2020pathological} comes from the fact that we get better estimates when estimating the derivatives and the $L_x^{\infty}$ norm of $\rho_{\varepsilon}$ instead of those of $v_n$ when $\varepsilon$ is large compared to $\frac{1}{n}$.
\begin{proof}
We first establish the upper bounds.
\subsubsection{Upper bounds}
Let $T_n(f):=f(n\cdot)$ be the scaling operator.
We have
\begin{align*}
v_n^{\varepsilon}(0,x)
	&=\kappa_n n^{\frac{3}{2}-s} \int\varphi(n(x-x'))\frac{1}{\varepsilon^3}\rho\left(\frac{x'}{\varepsilon}\right)\d x'\\
	&=\frac{\kappa_n n^{\frac{3}{2}-s}}{\varepsilon^3}(T_n\varphi\ast T_\frac{1}{\epsilon}\rho)(0,x),
\end{align*}
so that when $\alpha=(\alpha_1,\alpha_2,\alpha_3)\in\N^3$, denoting $m=\abs{\alpha}=\alpha_1+\alpha_2+\alpha_3$,
\begin{align*}
\partial^{\alpha}v_n^{\varepsilon}(0,x)
	&=\kappa_n n^{\frac{3}{2}-s} n^{m} \int T_n(\partial^{\alpha}\varphi)(x-x')\frac{1}{\varepsilon^3}\rho\left(\frac{x'}{\varepsilon}\right)\d x'	\\
	&=\frac{\kappa_n n^{\frac{3}{2}-s}}{\varepsilon^3}n^{m}(T_n(\partial^\alpha\varphi)\ast T_\frac{1}{\epsilon}\rho)(0,x).
\end{align*}
The Young's convolution inequality implies
\[
\|\partial^{\alpha}v_n^{\varepsilon}(0)\|_{L^{\infty}}
	\lesssim \kappa_n n^{\frac{3}{2}-s} n^{m}\|T_n(\partial^{\alpha}\varphi)\|_{L^{\infty}}\|\rho_{\varepsilon}\|_{L^1}
	\lesssim \lambda_n^{\frac{2}{p-1}} n^{m},
\]
\[
\|\partial^{\alpha}v_n^{\varepsilon}(0)\|_{L^{2}}
	\lesssim \kappa_n n^{\frac{3}{2}-s} n^{m}\|T_n(\partial^{\alpha}\varphi)\|_{L^{2}_{x}}\|\rho_{\varepsilon}\|_{L^1}
	\lesssim \kappa_n n^{m-s}\,.
\]
But one can also derivate the $\rho$ part
\begin{align*}
\partial^{\alpha}v_n^{\varepsilon}(0,x)
	&=\kappa_n n^{\frac{3}{2}-s} \left(\frac{1}{\varepsilon}\right)^{m} \int T_n(\varphi)(x')\frac{1}{\varepsilon^3}(\partial^{\alpha}\rho)\left(\frac{x-x'}{\varepsilon}\right)\d x' \\
	&=\frac{\kappa_n n^{\frac{3}{2}-s}}{\varepsilon^3}\left(\frac{1}{\varepsilon}\right)^{m}(T_n(\varphi)\ast T_{\frac{1}{\epsilon}}(\partial^\alpha\rho)(0,x),
\end{align*}
so that there also holds
\[
\|\partial^{\alpha}v_n^{\varepsilon}(0)\|_{L^{\infty}}
	\lesssim \kappa_n n^{\frac{3}{2}-s}\left(\frac{1}{\varepsilon}\right)^{m}\|T_n(\varphi)\|_{L^{1}}\frac{1}{\varepsilon^3}\|T_\frac{1}{\epsilon}(\partial^{\alpha}\rho)\|_{L^\infty}
	\lesssim \lambda_n^{\frac{2}{p-1}} \left(\frac{1}{\varepsilon}\right)^{m}\frac{1}{(\varepsilon n)^3},
\]
\begin{align*}
\|\partial^{\alpha}v_n^{\varepsilon}(0)\|_{L^{2}}
	&\lesssim \kappa_n n^{\frac{3}{2}-s} \left(\frac{1}{\varepsilon}\right)^{m}\|T_n(\varphi)\|_{L^{1}}\frac{1}{\varepsilon^3}\|T_\frac{1}{\epsilon}(\partial^{\alpha}\rho)\|_{L^2}\\
	&\lesssim\kappa_n n^{\frac{3}{2}-s} \left(\frac{1}{\varepsilon}\right)^{m-\frac{3}{2}}\frac{1}{(\varepsilon n)^3}\\
	&=\kappa_n n^{m-s}\frac{1}{(\varepsilon n)^{m+\frac{3}{2}}}\,.
\end{align*}
Note that when $m\leq 1$, we can still improve this estimate using that
\begin{equation*}
\|\partial^{\alpha}v_n^{\varepsilon}(0)\|_{L^{2}}
	\lesssim \kappa_n n^{\frac{3}{2}-s} \left(\frac{1}{\varepsilon}\right)^{m}\|T_n(\varphi)\|_{L^{2}}\frac{1}{\varepsilon^3}\|T_\frac{1}{\epsilon}(\partial^{\alpha}\rho)\|_{L^1}\\
	\lesssim \kappa_n n^{m-s}\frac{1}{(\varepsilon n)^3}\,.
\end{equation*}

These two observations sum up as
\[
\|\partial^{\alpha}v_n^{\varepsilon}(0)\|_{L^{\infty}}
	\lesssim \lambda_n^{\frac{2}{p-1}} n^{m}\min\left(1,\frac{1}{\varepsilon n}\right)^{m+3},
\]
\[
\|\partial^{\alpha}v_n^{\varepsilon}(0)\|_{L^{2}}
	\lesssim \kappa_n n^{m-s}\min\left(1,\frac{1}{\varepsilon n}\right)^{m+\frac{3}{2}}\,.
\]

Next, we deduce the estimates at time $t$. Using $V$ being bounded by $1$, we have
\[
\|v_n^{\varepsilon}(t)\|_{L^{\infty}}
	\lesssim \lambda_n^{\frac{2}{p-1}} \min\left(1,\frac{1}{\varepsilon n}\right)^{3},
\]
\[
\|v_n^{\varepsilon}(t)\|_{L^{2}}
	\lesssim \kappa_n n^{-s} \min\left(1,\frac{1}{\varepsilon n}\right)^{\frac{3}{2}}\,.
\]
Finally, we compute 
\begin{equation}\label{ineq:maj_dx_vn}
|\nabla v_n^{\varepsilon}(t)|
	\lesssim t|\abs{v_n^{\varepsilon}(0)}^{p-1} \nabla v_n^{\varepsilon}(0) V'(t|v_n^{\varepsilon}(0)|^{p-1})|+|\nabla v_n^{\varepsilon}(0)V(t|v_n^{\varepsilon}(0)|^{p-1})|,
\end{equation}
and
\begin{equation*}
    \begin{split}
        |\Delta v_n^{\varepsilon}(t)|
	\lesssim & t^2 |\abs{v_n^{\varepsilon}(0)}^{2p-3} (\nabla v_n^{\varepsilon}(0))^2 V''(t|v_n^{\varepsilon}(0)|^{p-1})| \\
	&+2 t |\abs{v_n^{\varepsilon}(0)}^{p-2}\abs{\nabla v_n^{\varepsilon}(0)}^2 V'(t|v_n^{\varepsilon}(0)|^{p-1})|\\
	&+t|\abs{v_n^{\varepsilon}(0)}^{p-1}\Delta v_n^{\varepsilon}(0)V'(t |v_n^{\varepsilon}(0)|^2)|\\
	&+|\Delta v_n^{\varepsilon}(0)V(t|v_n^{\varepsilon}(0)|^{p-1})|\,.
    \end{split}
\end{equation*}
Iterating, we obtain a similar formula when $\alpha\in\N^m$ is arbitrary, the leading term being the one where all the derivatives fall on the ODE profile $V$
\[
\abs{\partial^\alpha v_n^{\varepsilon}(t)}\lesssim \abs{\abs{v_n^\epsilon(0)}^{(p-1)m}\nabla v_n^\epsilon(0)}^m \abs{v_n^\epsilon(0)}t^m \abs{V^{(m)}\parent{t\abs{v_n^\epsilon(0)}^{p-1}}}\,.
\]
Since $|V|$ and its derivatives are bounded by $1$, we deduce that
\[
\|\nabla v_n^{\varepsilon}(t)\|_{L^{\infty}}\lesssim \lambda_n^{\frac{2}{p-1}} n (1+t\lambda_n^2) \min\left(1,\frac{1}{\varepsilon n}\right)^{1+3},
\]
\[
\|\Delta v_n^{\varepsilon}(t)\|_{L^{\infty}}
	\lesssim \lambda_n^{\frac{2}{p-1}} n^2(1+t\lambda_n^2 +(t\lambda_n^2)^2) \min\left(1,\frac{1}{\varepsilon n}\right)^{2+3}\,.
\]
On the $L^2$-based spaces, we also get
\[
\|\nabla v_n^{\varepsilon_n}(t)\|_{L^2}
	\lesssim \kappa_n n^{1-s}(1+t\lambda_n^2)\min\left(1,\frac{1}{\varepsilon n}\right)^{1+\frac{3}{2}},
\]
\[
\|\Delta v_n^{\varepsilon}(t)\|_{L^2}
	\lesssim \kappa_n n^{2-s}(1+(t\lambda_n^2)^2) \min\left(1,\frac{1}{\varepsilon n}\right)^{2+\frac{3}{2}}\,.
\]
The other estimates follow similarly.

\subsubsection{Lower bound} We now establish the lower bound $\|v_n^{\varepsilon}(t_n)\|_{H^s}\geq c \kappa_n(t_n\lambda_n^2)^s$  for every $\varepsilon\leq\varepsilon_n$, where we recall that $t_n\lambda_n^2\geq 1$ when $n\gg 1$. Fix $\varepsilon\leq \varepsilon_n$. When $0\leq s<1$, we have from interpolation that
\[
\|v_n^{\varepsilon}(t)\|_{H^{1}}^{2-s}
	\lesssim \|v_n^{\varepsilon}(t)\|_{H^{s}}\|v_n^{\varepsilon}(t)\|_{H^{2}}^{1-s}\,.
\]
Using the upper bound 
\[
\|v_n^{\varepsilon}(t)\|_{H^{2}}\lesssim \kappa_n n^{2-s}(1+(t\lambda_n^2)^2),
\]
it is enough to establish the lower bound 
\[
\|v_n^{\varepsilon}(t_n)\|_{H^{1}}\gtrsim  \kappa_n(t_n \lambda_n^2 ) n^{1-s}\,.
\]
Since the second term on the right-hand side in~\eqref{ineq:maj_dx_vn} has the upper bound
\[\|\nabla v_n^{\varepsilon}(0)V(t_n|v_n^{\varepsilon}(0)|^{p-1})\|_{L^2}\lesssim \kappa_n n^{1-s},\]
it is sufficient to establish the lower bound for the leading term
\[
\|t_n \abs{v_n^{\varepsilon}(0)}^{p-1} \nabla v_n^{\varepsilon}(0)V'(t_n|v_n^{\varepsilon}(0)|^{p-1})\|_{L^2}\,.
\]
Since $V'$ is of modulus one,
\begin{align*}
\|t_n \abs{v_n^{\varepsilon}(0)}^{p-1} \nabla v_n^{\varepsilon}(0)V'(t_n|v_n^{\varepsilon}(0)|^{p-1})\|_{L^2}
	&=t_n\|\abs{v_n^{\varepsilon}(0)}^{p-1} \nabla v_n^{\varepsilon}(0)\|_{L^2}\\
	&=t_n\lambda_n^{2+\frac{2}{p-1}} n\|\abs{T_n(\varphi)\ast\rho_{\varepsilon}}^{p-1}\left(T_n(\nabla \varphi)\ast\rho_{\varepsilon}\right)\|_{L^2}\\
	&=t_n\lambda_n^{2+\frac{2}{p-1}} n^{1-\frac{3}{2}}\|\abs{\varphi\ast\rho_{n\varepsilon}}^{p-1}\left(\nabla \varphi\ast\rho_{n\varepsilon}\right)\|_{L^2},
\end{align*}
the latter equality coming from the change of variable
\begin{align*}
T_nf\ast\rho_{\varepsilon}(x)
	&=T_n(f\ast\rho_{n\varepsilon}).
\end{align*}
Thanks to the choice of $\varepsilon$ such that $n\varepsilon\leq n\varepsilon_n= \frac{1}{100}$, the $L^2$ norm of $\varphi\ast\rho_{n\varepsilon}$ admits a lower bound that does not depend on $n$. Moreover, it is bounded from below by $c_0>0$ since $\varphi\ast\rho_{n\varepsilon}$ and its derivative tends to $\varphi$ and its derivative in the $L^p$ spaces when $n\varepsilon$ goes to $0$. 
To conclude, it only remains to observe that 
\[
t_n\lambda_n^{2+\frac{2}{p-1}} n^{1-\frac{3}{2}}=\kappa_n (t_n\lambda_n^2) n^{1-s}\,.
\]

When $1\leq s<\frac{3}{2}$, the idea is the same except that we need to interpolate between higher order Sobolev spaces
\[
\|v_n^{\varepsilon}(t)\|_{H^{2}}^{3-s}
	\lesssim \|v_n^{\varepsilon}(t)\|_{H^{s}}\|v_n^{\varepsilon}(t)\|_{H^{3}}^{2-s}\,.
\]

The leading term in the expression for $\|v_n^{\varepsilon}(t)\|_{H^{2}}$ is
\begin{equation*}
    \begin{split}
        \|t_n^2 \abs{v_n^{\varepsilon}(0)}^{2p-3} \abs{\nabla v_n^{\varepsilon}(0)}^2V''(t_n|v_n^{\varepsilon}(0)|^{p-1})\|_{L^2}
	&=t_n^2\|\abs{v_n^{\varepsilon}(0)}^{2p-3} (\nabla v_n^{\varepsilon}(0))^2\|_{L^2}\\
	&\hspace{-35pt}=t_n^2\lambda_n^{\frac{2(2p-1)}{p-1}} n^2\|\abs{T_n(\varphi)\ast\rho_{\varepsilon}}^{2p-3}\left(T_n(\nabla \varphi)\ast\rho_{\varepsilon}\right)^2\|_{L^2}\\
	&\hspace{-35pt}=t_n^2\lambda_n^{4+\frac{2}{p-1}} n^{2-\frac{3}{2}}\|\abs{\varphi\ast\rho_{n\varepsilon}}^{2p-3}\left(\nabla \varphi\ast\rho_{n\varepsilon}\right)^2\|_{L^2}\,,
    \end{split}
\end{equation*}
where $t_n^2\lambda_n^{4+\frac{2}{p-1}} n^{2-\frac{3}{2}}=\kappa_n (t_n\lambda_n^2)^2 n^{2-s}$.
To conclude, it only remains to use the upper bound
\[
\|v_n^{\varepsilon}(t_n)\|_{H^{3}}\lesssim \kappa_n n^{3-s}(t_n\lambda_n^2)^3\,.\qedhere
\]
\end{proof}

\subsection{Pathological set}\label{sec:pathological}
The pathological set is defined as
\begin{equation}
\label{eq:patho}
\mathcal{P}:=\{f\in H^s(\R^3)\mid \underset{\epsilon,t\to0}{\limsup}\,\|\Phi(t)(\rho_\epsilon\ast f)\|_{H^s}\to\infty\}\,.
\end{equation}
In this part, we prove that the pathological contains a dense $G_\delta$ set, that stems from the {\it ``tanghuru''} construction~\cite{SunTzvetkov2020pathological} consisting in the superposition of inflating bubbles. We fix $a\gg1$. For arbitrary $k\in\N$, $x_k\in\R^3$, we set $n_k:=e^{a^k}$, and define the $k$-th bubble centered at $x_k\in\R^3$, and concentrated at scale $n_k^{-1}$ $v_{0,k}$, writes
\[
v_{0,k}(x):=v_{n_k}(0,x-x_k)
	=\log(n_k)^{-\gamma}n_k^{\frac{3}{2}-s}\varphi(n_k(x-x_k))\,.
\]
\begin{rk}[Position of the bubbles]
In the construction from~\cite{SunTzvetkov2020pathological}, the $k$-th bubble concentrates in the point $x_k=\frac{1}{k}$, so that the distance between two consecutive bubbles is much larger than the scale of the bubble. Namely, $\abs{x_k-x_{k-1}}\sim \frac{1}{k^2}\gg e^{-a^{k}}$. Hence, the different bubbles are far enough and do not interact in short time, thanks to the finite propagation speed. Since we do not make use of this, we take arbitrary $x_k$ in the construction. Note that when $x_k=0$ for all $k$, the construction is radial and provides a pathological set even in the radial case.
\end{rk}

\begin{mydef}[Dense subset of the pathological set]
We denote by $S$ the set of initial data $f_0$ that can be decomposed under the form
\[
    f_0=u_0+\sum_{k=k_0}^{\infty}v_{0,k}\,, \quad k_0\geq 1\,,\quad u_0\in \classeC_c^{\infty}(\R^3)\,.
\]
\end{mydef}
Our aim is to show that the pathological set contains $S$, which is dense in $H^s(\R^3)$, and a $G_\delta$ subset as soon as every smooth solution to~\eqref{eq:NLS} is global. For this, we first establish upper bounds for the series in the definition of $S$, after regularization by convolution. Recall that we have set up the parameters in~\eqref{eq:parameters}, and, given $k\in\N$, we may denote for simplicity
\[
n = n_k = e^{a^k}\,,\quad\epsilon_k = \epsilon_{n_k}\,,\quad t_k = t_{n_k}\,,
\]
with $a\gg1$ to be chosen later in the analysis. At time $t_{n_k}$ and at scale $\epsilon_{n_k}$ we expect the $k$-th bubble to be responsible for the norm inflation, while the other bubbles concentrated at the larger scales $l^{-1}> k^{-1}$ and at smaller scales $l^{-1}< k^{-1}$ do not inflate.

\begin{lem}[$H^m$-bounds for initial data in the pathological set]
\label{lem:scale}
Let $k_0\geq1$, $k\geq k_0$, and denote $\epsilon_k=(100n_k)^{-1}$, with $n_k=e^{a^k}$\,. If $m<s$, there holds
\begin{align}
\label{eq:L2-k-1}
    \sum_{l=k_0}^{k-1}\norm{\rho_{\epsilon_k}\ast v_{0,l}}_{H^m(\R^3)}&\lesssim 1 \,,\\
\label{eq:L2-k+1}
    \sum_{l=k+1}^{\infty}\norm{\rho_{\epsilon_k}\ast v_{0,l}}_{H^m(\R^3)}&\lesssim n_{k+1}^{m-s} \left(\frac{n_k}{n_{k+1}}\right)^\frac{3}{2}\,.
\end{align}
Moreover, for every $m>s$, we have
\begin{align}
\label{eq:Hm-k-1}
    \sum_{l=k_0}^{k-1}\norm{\rho_{\epsilon_k}\ast v_{0,l}}_{H^m(\R^3)}&\lesssim n_{k-1}^{m-s} \,,\\
    \label{eq:Hm-k+1}
    \sum_{l=k+1}^{\infty}\norm{\rho_{\epsilon_k}\ast v_{0,l}}_{H^m(\R^3)}&\lesssim n_k^m n_{k+1}^{-s} \left(\frac{n_k}{n_{k+1}}\right)^\frac{3}{2}\,.
\end{align}
\end{lem}

\begin{rk}
\label{rem:eps}
Note that the estimate for the sum of the profiles concentrated at scale smaller than $n_k$ is only valid when $\varepsilon\geq\varepsilon_k$. As a consequence, we cannot evidence the norm inflation along the whole sequence of regularized initial data $(\rho_\epsilon\ast f_0)_{\epsilon>0}$, but only along the subsequence $(\rho_{\epsilon_k}\ast f_0)_{k\geq k_0}$, with $\epsilon_k=(100n_k)^{-1}$. For this reason, we only prove~\eqref{eq:limsup} instead of 
\begin{equation}
\label{eq:lim}
    \underset{T,\epsilon\to0}{\lim}\norm{u^\epsilon}_{L_t^\infty(\intervalcc{0}{T};H^s(\R^3))} =\infty\,.
\end{equation}
The norm inflation mechanism was also exhibited along a subsequence in  ~\cite{burq-tzvetkov-2008I,BurqGerardTzvetkov2005multilinear,Xia2021generic}. However, thanks to the finite propagation speed,~\cite{SunTzvetkov2020pathological} proved that the norm inflation for energy-subcritical nonlinear wave equations occurs along the whole sequence of regularized data. Nevertheless, defining the pathological set with the convergence type~\eqref{eq:limsup} is enough for our purpose. Indeed, in the context of probabilistic well-posedness, the convergence of the regularized strong solutions to the local solution from a randomized initial data holds for any sequence $\epsilon\to0$. Therefore, the pathological set and the good set of randomized initial data are disjoint.
\end{rk}

\begin{proof}
 When $m<s$, the large bubbles (i.e. with small index $l$ compared to $k$) contribute the most to the mass of the initial data.
\begin{equation*}
    \begin{split}
\sum_{l=k_0}^{k-1}\norm{\rho_{\epsilon_k}\ast v_{0,l}}_{H^m}
\lesssim \sum_{l=k_0}^{k-1}\kappa_{n_l} n_l^{m-s}
\lesssim n_{k_0}^{m-s}\,,
    \end{split}
\end{equation*}
whereas
\begin{equation*}
 \sum_{l=k+1}^\infty\norm{\rho_{\epsilon_k}\ast v_{0,l}}_{H^m} \lesssim
	 \sum_{l=k+1}^\infty \kappa_{n_l}n_l^{m-s}\left(\frac{n_k}{n_l}\right)^{\frac{3}{2}}
 	\lesssim n_{k+1}^{m-s}\left(\frac{n_k}{n_{k+1}}\right)^{\frac{3}{2}}\,.
\end{equation*}
This proves~\eqref{eq:L2-k-1} and~\eqref{eq:L2-k+1}, respectively. Next, we assume that $m>s$ and we establish~\eqref{eq:Hm-k-1} and~\eqref{eq:Hm-k+1}. It follows from Lemma~\ref{lem:growth_vn} that
\begin{equation*}
            \sum_{l=k_0}^{k-1}\|\rho_{\varepsilon_k}\ast v_{0,l}\|_{H^{m}}
	\lesssim\sum_{l= k_0}^{k-1}\kappa_{l} n_{k_l}^{m-s}\min\left(1,\frac{1}{\varepsilon_k n_l}\right)^{m+\frac{3}{2}}\lesssim n_k^{m-s}\,,
\end{equation*}
Similarly,
\begin{equation*}
\sum_{l= k+1}^{\infty}\kappa_{n_l} n_l^{m-s}\min\left(1,\frac{1}{\varepsilon_k n_l}\right)^{m+\frac{3}{2}}
	\lesssim \left(\frac{1}{\varepsilon_k}\right)^{m}\sum_{l= k+1}^{\infty}\kappa_{n_l}n_{l}^{-s}(\varepsilon_kn_l)^{-\frac{3}{2}}\\
	\lesssim n_k^{m}n_{k+1}^{-s}\parent{\frac{n_k}{n_{k+1}}}^{\frac{3}{2}}\,.
\end{equation*}
\end{proof}

\subsection{Perturbative analysis}\label{sec:perturbative-analysis}
In this part, we fix  $f_0\in S$ and consider the solution  $u^{\epsilon_k}$ to~\eqref{eq:NLS} with initial data $f_0\ast\rho_{\epsilon_k}$. We prove that in the scaling-supercritical regime $s<s_c$, and up to the time $t_{n_k}$, the nonlinear part of $u^{\epsilon_k}$ dominates the dispersion. Specifically, we prove that $u^{\epsilon_k}$ behaves like the ODE profile $v_k^{\epsilon_k}$, which is at scale ${n_k}$. Therefore, $u^{\epsilon_k}$ experiences a norm inflation at time $t_{n_k}$ as well. To do so, we perform a perturbative analysis by proving some a priori energy estimates on the difference between the ODE profile and $u^\epsilon_k$.

As a byproduct of the analysis, we can deduce from the a priori energy estimates combined with a standard continuity argument that the regularized solution $u^{\epsilon_k}$ exists in $C(\intervalcc{0}{t_{n_k}};H^2(\R^3))$.

\begin{prop}[Perturbative analysis]
\label{prop:perturbative} Let $k\geq 1$, $n_k=e^{a^k}$, $\varepsilon_k:=(100n_k)^{-1}$. When $s<s_c$, the local solution $u^{\epsilon_k}$ to~\eqref{eq:NLS} with initial data $f_0\ast \rho_{\epsilon_k}$ can be extended in $H^2(\R^3)$ up to time $t_{n_k}$. Moreover, there exists $C>0$ such that for $k$ large enough, 
\begin{equation}
\label{eq:pert}
    \norm{u^{\epsilon_k}(t_{n_k})-v_n^{\epsilon_k}(t_{n_k})}_{H^s(\R^3)}\leq C\,.
\end{equation}
\end{prop}
\begin{cor}[Norm inflation]
\label{cor:inflation}
For all $0<s<s_c$ there exist $0<c\ll1$ and a suitable choice of the parameters $\beta,\gamma$, such that for any $k$ large enough we have
\[
\norm{u^{\epsilon_k}(t_{n_k})}_{H^s} \geq c\log(n_k)^c\,.
\]
\end{cor}
Corollary~\ref{cor:inflation}, from which we deduce Theorem~\ref{thm-NLS}, follows from a priori estimate~\eqref{eq:pert} in the small-dispersion analysis, combined with the lower bound on $v_n^{\epsilon_k}$ from Lemma~\ref{lem:growth_vn}
\[
\norm{u^{\epsilon_k}(t_{n_k})}_{H^s} \geq c\log(n)^{s(\beta-\gamma)(p-1)-\gamma} - C\,.
\]
To prove Proposition~\ref{prop:perturbative}, we perform a standard perturbative analysis inspired from~\cite{BurqGerardTzvetkov2005multilinear} and~\cite{SunTzvetkov2020pathological}, using semi-classical energy estimates. The main difference is the presence of a linear correction, 
\[u_L^{\epsilon_k} :=\e^{it\Delta}u_0^{\epsilon_k}+\sum_{l=k_0}^{k-1}\e^{it\Delta}v_{0,l}^{\epsilon_k}\,.
\]
This correction ensures the smallness of the $L^2$-norm of $w(0)$. The correction is quite harmless up to time $t_{n_k}$, since $u_L^{\epsilon_k}$ concentrates at scale at most $(n_{k-1})^{-1}$, which is much larger than the scale $(n_k)^{-1}$ of the inflating bubble provided $a$ is chosen large enough in the construction. 

\begin{rk}
We remove the linear evolution of the large bubbles in $u_L^{\epsilon_k}$, in order to have an initial upper bound for the mass of order at most $n_k^{-s}$ initially. Indeed, as a consequence of Lemma~\ref{lem:scale} we may have 
\[
\norm{\sum_{l=k_0}^{k-1}v_{0,l}^{\epsilon_k}}_{L^2} \sim 1\,.
\]
Removing this contribution would deteriorate a bit the estimates on the Sobolev norms $H^m$ of $u_L^{\epsilon_k}$, that will be of size $n_{k-1}^m$ instead of~$C$. Nevertheless, since we run the argument up to time $t_{n_k}\ll n_{k-1}^m$ we shall be able to close the estimates as explained above.
\end{rk}

\begin{rk}
\label{rem:lower-bound}
In the case of the wave equation~\cite{SunTzvetkov2020pathological}, the norm inflation occurs much later, at time $t_{n_k}(NLW)\sim \lambda_{n_k}^{-1}$, since it occurs at time $t_{n_k}(NLS)\sim \lambda_{n_k}^{-2}$ for~\eqref{eq:NLS}, respectively. Consequently, a lower bound on $s$ is needed in~\cite{SunTzvetkov2020pathological} to control the growth of some source terms (containing $u_L^{\epsilon_k}$) up to time $t_{n_k}(NLW)\gg t_{n_k}(NLS)$. Nevertheless, it should be possible to remove this condition by using a refined nonlinear correction instead of $u_L^{\epsilon_k}$ in the perturbative analysis.
\end{rk}
\begin{proof}[Proof of Proposition~\ref{prop:perturbative}]
 By using the local Cauchy theory for~\eqref{eq:NLS} in $H^\sigma$ for $\sigma>\frac{3}{2}$, we see that a lower bound on the time of existence of $u^\epsilon_k$ in say $H^2$ follows from an a priori bound on its $H^2$-norm. We deduce such an a priori bound from  energy estimates. To do so, we write 
\[
w := u^{\epsilon_k} - u_L^{\epsilon_k} -v_n^{\epsilon_k}\,.
\]
Using the definition of the linear correction $u_L^{\epsilon_k}$, we see that initially,
\begin{equation}
    \label{eq:w0}
    w(0) = \sum_{l=k+1}^\infty v_{0,l}\ast\rho_{\epsilon_k}\,.
\end{equation}
We shorten the notation by denoting $n:=n_k$, and we consider the following semi-classical energy
\[
E_n(t) = \parentbig{n^{2s}\norm{w}_{L^2}^2 + n^{2(s-2)}\norm{w}_{H^2}^2}^\frac{1}{2}\,. 
\]
From the Gagliardo-Nirenberg inequality and the interpolation, we have that for $\sigma\in\intervalcc{0}{2}$,
\begin{equation}
    \label{eq:GN-E}
    \norm{w(t)}_{L^\infty}\lesssim n^{\frac{3}{2}-s}E_n(t)\,,\quad \norm{w(t)}_{H^\sigma} \lesssim n^{\sigma-s}E_n(t)\,.
\end{equation}
In particular, 
\(
\norm{w(t)}_{H^s} \lesssim E_n(t)\,.
\)
Moreover, the $H^s$-norm of the linear correction is bounded uniformly in $n$ since
\begin{equation}\label{eq:uL_Hs}
    \begin{split}
        \norm{u_L^{\varepsilon_k}}_{L_t^\infty H_x^s} 
        \leq \norm{u_0^{\varepsilon_k}}_{H^s} + \sum_{l=k_0}^{k-1}\norm{v_{0,l}^{\epsilon_k}}_{L^2}^{1-s}\norm{v_{0,l}^{\epsilon_k}}_{H^1}^s 
        \lesssim \norm{u_0}_{H^s} + \sum_{k=k_0}^{k-1}\kappa_{l}
        \lesssim 1\,.
    \end{split}
\end{equation}
For these reasons, inequality~\eqref{eq:pert} can be reduced to prove the a priori energy bound
\begin{equation*}
    \underset{0\leq t\leq t_n}{\sup} E_n(t) \leq 1\,.
\end{equation*}

Initially, we have from~\eqref{eq:L2-k+1} and~\eqref{eq:Hm-k+1} that
\[
\norm{w(0)}_{L^2}\leq C n_{k+1}^{-s}\parent{\frac{n}{n_{k+1}}}^\frac{3}{2}\,,\quad \norm{w(0)}_{H^2} \leq C n^2n_{k+1}^{-s}\parent{\frac{n}{n_{k+1}}}^\frac{3}{2}\,,
\]
so that
\begin{equation}\label{ineq:E0}
E_n(0)^2 = n^{2s}\norm{w(0)}_{L^2}^2+n^{-2(s-2)}\norm{\Delta w(0)}_{L^2}^2 \leq 2C\parent{\frac{n}{n_{k+1}}}^{3+2s} = Cn^{-c_0}\,,
\end{equation}
with $c_0:=(a-1)(3+2s)$. Now, we run a bootstrap argument to prove that $E_n(t)<1$ for all $t\leq t_n$. For shortness of notation, we drop the complex conjugation sign. By construction, $w$ is solution to the perturbed Schrödinger equation
\begin{equation*}
    (i\partial_t+\Delta)w = \abs{w+u_L^{\epsilon_k}+v_n^{\epsilon_k}}^{p-1}(w+u_L^{\epsilon_k}+v_n^{\epsilon_k}) - \abs{v_n^{\epsilon_k}}^{p-1}v_n^{\epsilon_k} + \Delta v_n^{\epsilon_k} =: F + \Delta v_n^{\epsilon_k}\,,
\end{equation*}
with initial condition $w(0)$ given by~\eqref{eq:w0}, and where the forcing term is
\[
F = \parent{w+u_L^{\epsilon_k}}\bigo{\abs{w}^{p-1}+\abs{u_L^{\epsilon_k}}^{p-1}+\abs{v_n^{\epsilon_k}}^{p-1}}\,.
\]
Denoting $\Lambda := \Delta F$, we have
\begin{equation*}
    \begin{split}
        \abs{\Lambda} \lesssim &(\abs{\Delta w}+\abs{\Delta u_L^{\epsilon_k}})\parent{\abs{w}^{p-1}+\abs{u_L^{\epsilon_k}}^{p-1}+\abs{v_n^{\epsilon_k}}^{p-1}} \\
        &+ (\abs{\nabla w}+\abs{\nabla u_L^{\epsilon_k}})^2\parent{\abs{w}^{p-2}+\abs{u_L^{\epsilon_k}}^{p-2}+\abs{v_n^{\epsilon_k}}^{p-2}} \\
        &+ (\abs{\nabla w}+\abs{\nabla u_L^{\epsilon_k}})\abs{\nabla v_n^{\epsilon_k}}\parent{\abs{w}^{p-2}+\abs{u_L^{\epsilon_k}}^{p-2}+\abs{v_n^{\epsilon_k}}^{p-2}}\\
        &+(\abs{w|+|u_L^{\epsilon_k}})\abs{\nabla^2 v_n^{\epsilon_k}}\parent{\abs{w|^{p-2}+|u_L^{\epsilon_k}}^{p-2}+\abs{v_n^{\epsilon_k}}^{p-2}} \\
        &+(\abs{w|+|u_L^{\epsilon_k}})\abs{\nabla v_n^{\epsilon_k}}^2\parent{\abs{w|^{p-3}+|u_L^{\epsilon_k}}^{p-3}+\abs{v_n^{\epsilon_k}}^{p-3}}\,.
    \end{split}
\end{equation*}

First, we control the source term coming from derivatives of $\Delta v_n^{\epsilon_k}$. We observe that $2<(\frac{3}{2}-s)(p-1)$ when $s<s_c$.~\footnote{This seems to be the only place where we see the condition $s<s_c$.} Hence, we deduce from Lemma~\ref{lem:growth_vn} that 
\begin{multline}
\label{eq:incr-source}
    n^s\norm{\Delta v_n^{\epsilon_k}}_{L_{t_n}^\infty L_x^2} + n^{s-2}\norm{\Delta^2 v_n^{\epsilon_k}}_{L_{t_n}^\infty L_x^2} \leq C(t\lambda_n^2)^4n^2 \leq C\log(n)^{4(\beta-\gamma)(p-1)}n^2 \\
    \leq Cn^{(\frac{3}{2}-s)(p-1)-\delta_0}\,,
\end{multline}
with, say,
\[
\delta_0:=\frac{1}{2}\parent{ (\frac{3}{2}-s)(p-1)- 2} >0\,.
\]

Now, we control the source terms comings from $u_L^{\epsilon_k}$ and $v_n^{\epsilon_k}$ in the expression of $F$ and $\Delta F$. First, notice from~\eqref{eq:Hm-k-1} and from the $H^s$-bound~\eqref{eq:uL_Hs}  that for any $m\geq s$, there exists $C_m>0$ such that 
\begin{equation}\label{eq:est-uL}
\norm{u_L^{\epsilon_k}}_{L_t^\infty H^m_x} \leq C_mn_{k-1}^{m-s}\,,    
\end{equation}
whereas when $m<s$ we can only use that thanks to~\eqref{eq:uL_Hs}, we have $\norm{u_L^{\epsilon_k}}_{L_t^\infty H^m_x} \leq C$. 
Using the Gagliardo-Nirenberg inequality with $s<\frac{3}{2}$, we deduce that
\begin{equation*}
\norm{u_L^{\epsilon_k}}_{L_t^\infty L^{\infty}_x}
	\lesssim (\|u_L^{\epsilon_k}\|_{L^{\infty}_tH^s_x})^{\frac{1}{2(2-s)}} (\|	u_L^{\epsilon_k}\|_{L^{\infty}_tH^2_x})^{\frac{3-2s}{2(2-s)}}
	\lesssim n_{k-1}^{\frac 32-s}\,.    
\end{equation*}
In order to prevent another unnecessary loss later on, we also observe from the Sobolev embedding $\dot{H}^{\frac 34}\hookrightarrow L^4$ that
\[
\norm{\abs{\nabla}u_L^{\epsilon_k}}_{L^{\infty}_tL^4_x} 
	\lesssim  \norm{u_L^{\epsilon_k}}_{L^{\infty}_tH^{1+\frac 34}_x} \,,
\]
and since $1+\frac 34>\frac 32>s_c$, we deduce
\[
\norm{\abs{\nabla}u_L^{\epsilon_k}}_{L^{\infty}_tL^4_x} 
	\lesssim  n_{k-1}^{1+\frac 34 -s} \,.
\]

Using the estimates on the profile $v$ from Lemma~\ref{lem:growth_vn}, 
we can therefore control the source terms $u_L^{\epsilon_k}$ and $v$ up to time $t_n\sim \lambda_n^{-2}$, uniformly in $n$ and estimate
\begin{multline*}
    n^s\norm{F(t)}_{L^2} \lesssim n^s\norm{w}_{L^2}\parent{\norm{w}_{L^\infty}^{p-1} + \norm{u_L^{\epsilon_k}}_{L^\infty}^{p-1} +  \norm{v_n^{\epsilon_k}}_{L^\infty}^{p-1}} + \\
n^s\norm{u_L^{\epsilon_k}}_{L^\infty}\parent{\norm{w}_{L^2}\norm{w}_{L^\infty}^{p-2}+\norm{u_L^{\epsilon_k}}_{L^2}\norm{u_L^{\epsilon_k}}_{L^\infty}^{p-2}+\norm{v_n^{\epsilon_k}}_{L^2}\norm{v_n^{\epsilon_k}}_{L^\infty}^{p-2}}\,.
\end{multline*}
As a consequence, 
\begin{multline*}
    n^s\norm{F(t)}_{L^2}\lesssim E_n(t)\parent{ E_n(t)^{p-1}n^{(\frac{3}{2}-s)(p-1)}+ n_{k-1}^{(\frac{3}{2}-s)(p-1)}+ n^{(\frac{3}{2}-s)(p-1)}} \\
+n^sn_{k-1}^{\frac{3}{2}-s}\parent{n^{-s}E_n(t)n^{(\frac{3}{2}-s)(p-2)} + n_{k-1}^{(\frac{3}{2}-s)(p-2)} + n^{-s}n^{(\frac{3}{2}-s)(p-2)}}\,.
\end{multline*}
Supposing that $E_n(t)\ll 1$, which is the case at time $t=0$,  this leads to the estimate
\[
    n^s\norm{F(t)}_{L^2} \lesssim E_n(t)n^{(\frac{3}{2}-s)(p-1)}+n^{(\frac{3}{2}-s)(p-1)}\parent{\parent{\frac{n_{k-1}}{n}}^{\frac{3}{2}-s}+n^s\parent{\frac{n_{k-1}}{n}}^{(\frac{3}{2}-s)(p-1)}}\,.
\]
We recall that we choose $n=n_k=e^{a^k}$ for some $a\gg 1$, so that $\frac{n_{k-1}}{n} = n^{-\theta}$ with $\theta=1-\frac{1}{a}$. Therefore, there holds
\[
n^s\norm{F(t)}_{L^2} \lesssim E_n(t)n^{(\frac{3}{2}-s)(p-1)}+n^{(\frac{3}{2}-s)(p-1)}\parent{n^{-\theta(\frac{3}{2}-s)}+n^{s-\theta(\frac{3}{2}-s)(p-1)}}\,.
\]
We want to ensure that 
\begin{equation}
    \label{eq:c}
    s< \theta\left(\frac{3}{2}-s\right)(p-1)\,,
\end{equation}
so that there exist $C>0$ and $c>0$ such that 
\begin{equation}
    \label{eq:incr-mass}
    n^s\norm{F(t)}_{L^2} \leq C E_n(t)n^{\left(\frac{3}{2}-s\right)(p-1)}+n^{(\frac{3}{2}-s)(p-1)-c}\,.
\end{equation}
Since $s<s_c=\frac{3}{2}-\frac{2}{p-1}$, the condition~\eqref{eq:c} reduces to 
\[
s_c < \theta\left(\frac{3}{2}-s_c\right)(p-1)\,,
\]
which is satisfied as soon as $\theta$ is sufficiently close to $1$, that is when $a$ is chosen sufficiently large so that 
\[
\frac{3}{4}<1-\frac{1}{a}\,.
\]

Similarly, we estimate the increment of the $H^2$-norm as in~\cite{BurqGerardTzvetkov2005multilinear}, controlling the terms involving $v_n^{\epsilon_k}, \nabla v_n^{\epsilon_k}, \Delta v_n^{\epsilon_k}$ or $u_L^{\epsilon_k}, \nabla u_L^{\epsilon_k}, \Delta u_L^{\epsilon_k}$ in $L^\infty$:
\begin{equation*}
\begin{split}
    &\norm{\Lambda(t)}_{L^2} 
    \lesssim  \norm{\Delta w}_{L^2}\parent{\norm{u_L^{\epsilon_k}}_{L^\infty}^{p-1}+\norm{w}_{L^\infty}^{p-1}+\norm{v_n^{\epsilon_k}}_{L^\infty}^{p-1}} \\
    &\quad \quad \quad \quad   +\norm{\Delta u_L^{\epsilon_k}}_{L^\infty}\parent{\norm{w}_{L^2}\norm{w}_{L^\infty}^{p-2}+\norm{v_n^{\epsilon_k}}_{L^2}\norm{v_n^{\epsilon_k}}_{L^\infty}^{p-2}} +\norm{\Delta u_L^{\epsilon_k}}_{L^2}\norm{u_L^{\epsilon_k}}_{L^\infty}^{p-1} \\
    &+ \norm{\abs{\nabla} w}_{L^4}^2\parent{\norm{u_L^{\epsilon_k}}_{L^\infty}^{p-2}+\norm{w}_{L^\infty}^{p-2}+\norm{v_n^{\epsilon_k}}_{L^\infty}^{p-2}} \\
    &\quad \quad \quad \quad +\norm{\abs{\nabla} u_L^{\epsilon_k}}_{L^\infty}^2\parent{\norm{w}_{L^2}\norm{w}_{L^\infty}^{p-3}+\norm{v_n^{\epsilon_k}}_{L^2}\norm{v_n^{\epsilon_k}}_{L^\infty}^{p-3}}+ \norm{ \abs{\nabla} u_L^{\epsilon_k}}_{L^4}^2\norm{u_L^{\epsilon_k}}_{L^\infty}^{p-2} \\
    &+ \left(\norm{\abs{\nabla} w}_{L^2}\norm{\abs{\nabla} v_n^{\epsilon_k}}_{L^\infty}+ \norm{\abs{\nabla} u_L^{\epsilon_k}}_{L^\infty}\norm{\abs{\nabla} v_n^{\epsilon_k}}_{L^2}\right)\parent{\norm{u_L^{\epsilon_k}}_{L^\infty}^{p-2}+\norm{w}_{L^\infty}^{p-2}+\norm{v_n^{\epsilon_k}}_{L^\infty}^{p-2}} \\
    &+ \left(\norm{w}_{L^2}\norm{\Delta v_n^{\epsilon_k}}_{L^\infty}+ \norm{u_L^{\epsilon_k}}_{L^\infty}\norm{\Delta v_n^{\epsilon_k}}_{L^2}\right)\parent{\norm{u_L^{\epsilon_k}}_{L^\infty}^{p-2}+\norm{w}_{L^\infty}^{p-2} + \norm{v_n^{\epsilon_k}}_{L^\infty}^{p-2}}\\
    &+ \left(\norm{w}_{L^2}\norm{\abs{\nabla} v_n^{\epsilon_k}}_{L^\infty}^2+\norm{u_L^{\epsilon_k}}_{L^\infty}\norm{\abs{\nabla} v_n^{\epsilon_k}}_{L^2}\norm{\abs{\nabla} v_n^{\epsilon_k}}_{L^\infty}\right) \parent{\norm{u_L^{\epsilon_k}}_{L^\infty}^{p-3}+\norm{w}_{L^\infty}^{p-3} + \norm{v_n^{\epsilon_k}}_{L^\infty}^{p-3}}\,.
\end{split}
\end{equation*}
In addition to inequality~\eqref{eq:GN-E}, we shall use the Sobolev embedding $\dot{H}^\frac{3}{4}\hookrightarrow L^4$ and interpolation to get
\[
\norm{\abs{\nabla}w(t)}_{L^4} \leq n^{1+\frac{3}{4}-s}E_n(t)\,.
\]
Using the bounds on $v_n^{\epsilon_k}$ and on $u_L^{\epsilon_k}$, we obtain 
\begin{equation*}
\begin{split}
    n^{s-2}\norm{\Lambda(t)}_{L^2} 
    &\lesssim  E_n(t)\parent{n_{k-1}^{(\frac{3}{2}-s)(p-1)}+n^{(\frac{3}{2}-s)(p-1)}(E_n(t)^{p-1}+\kappa_n^{p-1})} \\
    &\quad\quad\quad\quad +n_{k-1}^{2+\frac 32-s}n^{-2+(\frac 32 -s)(p-2)}\left(E_n(t)^{p-1}+\kappa_n^{p-1}\right) + n^{s-2}n_{k-1}^{2-s+(\frac 32 -s)(p-1)}\\
    &+n^{\frac{3}{2}-s}E_n(t)^2\parent{n_{k-1}^{(\frac{3}{2}-s)(p-2)}+n^{(\frac{3}{2}-s)(p-2)}(E_n(t)^{p-2}+\kappa_n^{p-2})} \\
    &\quad\quad\quad\quad +n_{k-1}^{2+3-2s}n^{-2+(\frac 32 -s)(p-3)}\left(E_n(t)^{p-3}+\kappa_n^{p-3}\right)
    +n^{s-2}{ n_{k-1}^{2+\frac 32-2s+(\frac 32 -s)(p-2)}}\\
   &+\kappa_n\left(n^{\frac{3}{2}-s}E_n(t)+n_{k-1}^{1+\frac 32 -s}n^{-1}\right)\parent{n_{k-1}^{(\frac{3}{2}-s)(p-2)}+n^{(\frac{3}{2}-s)(p-2)}(E_n(t)^{p-2}+\kappa_n^{p-2})} \\
    &+\kappa_n\left(n^{\frac{3}{2}-s} E_n(t)+n_{k-1}^{\frac 32 -s}\right)\parent{n_{k-1}^{(\frac{3}{2}-s)(p-2)}+n^{(\frac{3}{2}-s)(p-2)}(E_n(t)^{p-2}+\kappa_n^{p-2})} \\
    &+\kappa_n^2\left(n^{2(\frac{3}{2}-s)}E_n(t)+ n_{k-1}^{\frac{3}{2}-s} n^{\frac{3}{2}-s}\right)
    \parent{n_{k-1}^{(\frac{3}{2}-s)(p-3)}+n^{(\frac{3}{2}-s)(p-3)}(E_n(t)^{p-3}+\kappa_n^{p-3})}\,.
\end{split}
\end{equation*}
Once again, we assume that $E_n(t)\ll1$. Noting that $\kappa_n\leq 1$ and that $n_{k-1}\leq n$, we use the bounds
\[
n_{k-1}^{(\frac{3}{2}-s)(p-i)}+n^{(\frac{3}{2}-s)(p-i)}(E_n(t)^{p-i}+\kappa_n^{p-i})
	\lesssim \max(1,n^{(\frac 32 -s)(p-i)})\,,\quad i\in\set{1,2,3,\dots}\,.
\]
We then simplify the above expression by writing 
\begin{equation*}
    \begin{split}
     n^{s-2}\norm{\Lambda(t)}_{L^2}
     \lesssim & E_n(t)n^{(\frac{3}{2}-s)(p-1)}\\
     &+n_{k-1}^{2+\frac 32-s}n^{-2+(\frac 32-s)(p-2)}+n^{s-2}n_{k-1}^{2-s+(\frac 32-s)(p-1)}\\
     &+n_{k-1}^{2+3-2s}n^{-2+(\frac 32-s)(p-3)}+n^{s-2}n_{k-1}^{2-s+(\frac 32-s)(p-1)} \\
     &+n_{k-1}^{1+\frac 32 -s}n^{-1+(\frac 32 -s)(p-2)}\\
     &+n_{k-1}^{\frac 32 -s}n^{(\frac 32 -s)(p-2)}\,.
    \end{split}
\end{equation*}
The cross-terms between $n_{k-1}$ and $n$ are equal to
\[
n^{(\frac 32-s)(p-1)}\left(\left(\frac{n_{k-1}}{n}\right)^{2+\frac 32 -s}
	+\left(\frac{n_{k-1}}{n}\right)^{2-s}
	+\left(\frac{n_{k-1}}{n}\right)^{2+3 -2s}
	+\left(\frac{n_{k-1}}{n}\right)^{1+\frac 32 -s}
	+\left(\frac{n_{k-1}}{n}\right)^{\frac 32 -s}\right).
\]
Since $\frac{n_{k-1}}{n}=n^{-\theta}$, there exist $C,c>0$ such that 
\begin{equation}
    \label{eq:incr-h2}
    n^{s-2}\norm{\Lambda(t)}_{L^2}\leq C E_n(t)n^{(\frac{3}{2}-s)(p-1)} + n^{(\frac{3}{2}-s)(p-1)-c\theta}\,.
\end{equation}
Combining~\eqref{eq:incr-source},~\eqref{eq:incr-mass} and~\eqref{eq:incr-h2} with the energy estimate 
\[
\frac{\d}{\d t} \left(E_n(t)^2\right)
	\leq 2E_n(t)\parent{n^s\norm{F(t)}_{L^2}+n^{s-2}\norm{\Lambda(t)}_{L^2}}+n^s\norm{\Delta v_n^{\epsilon_k}}_{L_{t_n}^\infty L_x^2} + n^{s-2}\norm{\Delta^2 v}_{L_{t_n}^\infty L_x^2} \,,
\]
we get that there exist $C,c>0$ such that
\[
\frac{\d}{\d t} E_n(t)^2 \leq CE_n(t)^2n^{(\frac{3}{2}-s)(p-1)} + Cn^{(\frac{3}{2}-s)(p-1)-c}= C\lambda_n^2 E_n^2 + C\lambda_n^2n^{-2c}\,.
\]
It follows from the Gronwall estimate that 
\[
    E_n(t)^2\leq \parent{E_n(0)^2 + Cn^{-2c}t_n\lambda_n^2}\exp\parent{Ct_n\lambda_n^2} \\
    \leq \parent{E_n(0)^2 + Cn^{-c}}\exp\parent{C\log(n)^{(\beta-\gamma)(p-1)}}\,.
\]
We recall that from~\eqref{ineq:E0}, we have $E_n(0)^2\leq Cn^{-c_0}$ for some $c_0>0$. Finally, we fix parameters $\beta,\gamma$ small enough such that $(p-1)\beta<\frac{1}{2}$ and $\gamma < \frac{1}{2}\beta$. Then, we fix $n_0\in\N$ such that for all $n\geq n_0$, for some  constant $C>0$ encountered in the analysis, we have
\begin{equation*}
C\log(n)^{(\beta-\gamma)(p-1)} \leq \frac{\delta}{2}\log(n) + C(\delta)\,,\quad \delta:=\min(c_0,c)>0\,,
\end{equation*}
so that
\[
E_n(t)^2\lesssim_\delta n^{-\delta} n^{\frac{\delta}{2}}\lesssim_\delta n^{-\frac{\delta}{2}} \ll1\,.
\]
We conclude by a standard bootstrap argument.
\end{proof}

We proved that $S\subset\mathcal{P}$. Finally, we show that $S$ is dense in $H^s$. Then, by using this inclusion, we deduce that $\mathcal{P}$ contains a dense $G_\delta$ set as soon as~\eqref{eq:NLS} is globally well-posed in $H^2$.

\begin{prop}[Density]\label{prop:dense}
For $g\in H^2(\R^3)$, we denote $\Phi(t)g$ the maximal solution to~\eqref{eq:NLS} in $C(\intervalcc{0}{T^*};H^2(\R^3))$ with initial data $g$. The pathological set $\mathcal{P}$ defined in~\eqref{eq:patho} is dense in $H^s(\R^3)$. Moreover, in the defocusing energy-subcritical and critical cases $p\in\set{3,5}$, where the Cauchy problem associated with~\eqref{eq:NLS} is globally well-posed in $H^2(\R^3)$, $\mathcal{P}$ contains a dense $G_{\delta}$ subset of $H^s(\R^3)$.
\end{prop}

\begin{proof}
By construction, we know that
\[
\sum_{k=k_0}^{\infty}\|v_{0,k}\|_{H^s}
	\lesssim \sum_{k=k_0}^{\infty}\kappa_{n_k}\longrightarroww{k_0\to\infty}{} 0,
\]
therefore $S$ is dense in $H^s(\R^3)$.

We now assume that global well-posedness in $H^2(\R^3)$ holds, and we follow the proof of Corollary~1.2 in~\cite{SunTzvetkov2020pathological}. Specifically, we prove that $\mathcal{P}$ contains a countable intersection of nonempty open sets. Let us consider the sequences $(\varepsilon_k)$ and $(t_k)$ going to zero as in~\eqref{eq:parameters}, and define
\[
\mathcal{O}_{N,k}:=\{f\in H^s(\R^3)\mid \|\Phi(t_{n_k})(\rho_{\varepsilon_k}\ast f)\|_{H^s}>N\}\,.
\]
By definition of the $\limsup$, we have
\[
\bigcap_{N=1}^{\infty}\bigcap_{k=k_1}^{\infty}\bigcup_{l=k_1}^{\infty}\mathcal{O}_{N,k}\subset S.
\]

We already know that $\bigcup_{k=k_1}^{\infty}\mathcal{O}_{N,k}$ is nonempty since it contains for instance the lollipop $\sum_{l=k_1}^{\infty}v_{0,l}$. 
 Therefore, it remains to show that for every $N$ and $k$, $\mathcal{O}_{N,k}$ is an open set. 

Let $f_0\in \mathcal{O}_{N,k}$, $t\leq t_{n_k}$ and $\delta>0$ to be chosen later. Using the Duhamel formula and the Cauchy theory in $H^2$, we know that there exists $C_0>0$ such that for every $f\in H^s$,
\[
\|\Phi(t)(\rho_{\varepsilon_k}\ast f)-\Phi(t)(\rho_{\varepsilon_k}\ast f_0)\|_{H^2}
	\leq C_0(\|\rho_{\varepsilon_k}\ast f\|_{H^2}^{p-1}+\|\rho_{\varepsilon_k}\ast f_0\|_{H^2}^{p-1})\|\rho_{\varepsilon_k}\ast (f-f_0)\|_{H^2}\,.
\]
Using the inequality 
\[
\|\rho_{\varepsilon}\ast f\|_{H^2}\lesssim \varepsilon^{-(2-s)}\|f\|_{H^s}\,,
\]
we deduce that for any $f\in H^s$, 
\[
\|\Phi(t)(\rho_{\varepsilon_k}\ast f)-\Phi(t)(\rho_{\varepsilon_k}\ast f_0)\|_{H^2}
	\leq C_0\varepsilon_k^{-(2-s)p}(\|f\|_{H^s}^{p-1}+\|f_0\|_{H^s}^{p-1})\|f-f_0\|_{H^s}\,.
\]
Since $\|\Phi(t)(\rho_{\varepsilon_k}\ast f_0)\|_{H^s}>N$, we deduce that when $f$ is at distance less than $\delta$ from $f_0$ in $H^s$, then $\|\Phi(t)(\rho_{\varepsilon_k}\ast f)\|_{H^s}>N$ for some suitable $\delta>0$, so that $\mathcal{O}_{N,k}$ is open. 
\end{proof}

\section{Generic well-posedness}
\label{sec:proba}

This section is devoted to the proof of Theorem~\ref{theo-lwp} for the cubic Schrödinger equation, at scaling-supercritical regularity $H^s(\R^3)$, i.e.\@~with $\frac{1}{4}<s<\frac{1}{2}$.~\footnote{The lower bound on $s$ stems from the contributions of the \emph{high-high-low} and of the \emph{high-low-low} interactions in the first Duhamel iteration. Besides, by refining the recentering method in~\cite{benyi-oh-pocovnicu-2019}, the regularity threshold has been lowered to $\frac{1}{5}<s$. In~\cite{dodson-luhrmann-mendelson-19}, a different approach based on local smoothing in lateral spaces improves the regularity threshold in dimension 4, and can certainly be adapted in dimension 3. Note also that $s$ needs to be strictly grater than the threshold because of the summation over dyadic blocks.} The technics shall be more specific to the Schrödinger semigroup and its dispersive properties. Namely, we make use of Strichartz and bilinear estimates in the framework of Fourier restriction spaces, that we recall in parts~\ref{sec:prelim} and~\ref{sec:proba-lwp}. We then adapt the proof of the almost sure local well-posedness result from~\cite{benyi2015-local} to establish the convergence of the regularized solutions to the solution obtained from the random initial data regularized by convolution, as claimed in Theorem~\ref{thm-NLS}.

In what follows, we fix $s$ a scaling-supercritical exponent such that $\frac{1}{4}<s<\frac{1}{2}$, and a scaling-subcritical (or critical) Sobolev space $H^\sigma$, with $\sigma$ satisfying
\[
s_c =\frac{1}{2}< \sigma < 2s\,.
\]

\subsection{Preliminaries}\label{sec:prelim}

We first recall the basic dispersive estimate for the free Schrödinger evolution.

\begin{lem}[Strichartz estimates]Given $2\leq p,q \leq \infty$ satisfying the admissibility condition
\begin{equation}
\label{eq:str-adm}
    \frac{2}{p} + \frac{3}{q} = \frac{3}{2}\,,
\end{equation}
there exists $C>0$ such that for all $u_0\in L^2(\R^3)$\,, we have
\[
\norm{\e^{it\Delta}u_0}_{L_t^p(\R ; L_x^q(\R^3)}\leq C\norm{u_0}_{L_x^2(\R^3)}\,.
\]
\end{lem}
We emphasize that the admissibility condition is imposed by the scaling invariance. A key feature of the randomization procedure is to rule out the scaling invariance, and to relax the admissibility condition, as written in Lemma~\ref{lem:prob-str}. The following bilinear estimate due to Bourgain indicates a nonlinear smoothing effect when high and low frequencies interact. Combined with the probabilistic Strichartz estimates, it is the main mechanism responsible for the smoothing effect observed in the Picard iteration from a randomized initial data.  
\begin{lem}[Bilinear estimate~\cite{Bourgain1998Bilinear}]
\label{lem:bili}
Given $N\ll M$ two dyadic integers and $u_0, v_0 \in L^2(\R^3)$ such that 
\[
\supp\,\mathcal{F}(u_0) \subset\set{\xi\in\R^3\mid\abs{\xi}\lesssim N}\,,\quad \supp\,\mathcal{F}(v_0) \subset\set{\xi\in\R^3\mid\frac{1}{2}M\leq\abs{\xi}\leq 2M}\,.
\]
We have 
\[
\norm{\parent{\e^{it\Delta}u_0}\parent{\e^{it\Delta}v_0}}_{L_{t,x}^2(\R\times\R^3)}\lesssim NM^{-\frac{1}{2}}\norm{u_0}_{L_x^2(\R^3)}\norm{v_0}_{L_x^2(\R^3)}\,.
\]
\end{lem}

Let us now set up the relevant randomization procedure in the Euclidean space, and recall the key probabilistic Strichartz estimates obtained in this context. 


\subsubsection{Wiener randomization}

The \emph{Wiener decomposition} of a function in $f_0\in L^2$ is defined from a unit-scale partition of unity of the frequency space given by a function $\psi\in\classeC_c^{\infty}(\R^3)$ supported in $\intervalcc{-2}{2}^3$, such that
\[
\sum_{k\in\Z^3} \psi(\xi-k) = 1\,.
\]
Defining the corresponding unit-scale Fourier multiplier $P_{1,k}=\mathcal{F}^{-1}(\psi(\cdot-k)\mathcal{F})$ on the unit cube centered around $k$, we decompose any function in $L^2(\R^3)$ into a sum of unit-blocks
\[
f_0 = \sum_{k\in\Z^3} P_{1,k}f_0\,.
\]
The randomization procedure consists in decoupling these individual blocks by some independent Gaussian variables. Specifically, given a function\footnote{In order to have a reasonable measure, we ask for some non-degeneracy conditions on $f_0$. Namely, we want to ensure that $\supp \mu$ is dense in $H^s$, and that $\mu(H^{\sigma}\setminus H^s)=0$ for all $\sigma>s$. To do so, it is sufficient to assume that $\norm{P_{1,k}f_0}_{L^2} > 0$ for all $k\in\R^3$. We refer to Lemma B.1 in~\cite{burq-tzvetkov-2008I} and~\cite{BenyiOhPocovnicuSurvey2019} for further details.} $f_0\in H^s(\R^3)$, a probability space $(\Omega, \mathcal{A}, \mathbb{P})$ and some independent complex Gaussian variables $(g_k)_{k\in\Z^3}$, we define a probability measure $\mu$ supported on $H^s(\R^3)$ induced by the random variable 
\[
\omega\in\Omega\longmapsto f_0^\omega := \sum_{k\in\Z^3} g_k(\omega)P_{1,k}f_0\in H^s(\R^3)\quad \text{a.s.}
\]
This randomization procedure is called the \emph{Wiener randomization}.

\subsubsection{Strichartz estimates}
The main idea is to combine the unit scale Bernstein estimate with probabilistic decoupling in order to improve the integrability bounds and to relax the Strichartz admissibility condition, for a large measure set of initial data. 

\begin{mydef}
Let $M\gg1$ be a parameter. We say that an event $A=A(M)$ occurs $M$-certainly when there exist some positive constants $C,c,\theta>0$ such that 
\[
\mathbb{P}(A(M))\geq 1 - C\exp\parent{-cM^\theta}\,.
\]
\end{mydef}

\begin{lem}[Probabilistic Strichartz estimates~\cite{benyi2015-local}, Proposition 1.4]
\label{lem:prob-str}
Fix a Strichartz admissible pair $(q,r)$, with $2\leq q,r<\infty$, satisfying the Strichartz admissibility condition~\eqref{eq:str-adm}. For any $\widetilde{r}$ such that  $r\leq\widetilde{r}<\infty$, there exist $C,c>0$ such that for all $\lambda>0$,
\begin{equation}
    \label{eq:ld-str}
    \mathbb{P}\parent{\set{\omega\mid\ \norm{\japbrak{\nabla}^s\e^{it\Delta}f_0^\omega}_{L_t^q(\R; L_x^{\widetilde{r}}(\R^3))}>\lambda}} \leq C\exp\parent{-c\lambda^2\norm{f_0}_{H^s}^{-2}}\,.
\end{equation}
As a consequence, for any $\lambda>0$, the following Strichartz estimate holds $\lambda$-certainly 
\begin{equation}
\label{eq:prob-str}
\norm{\japbrak{\nabla}^s\e^{it\Delta}f_0^\omega}_{L_t^q(\R; L_x^{\widetilde{r}}(\R^3))} \leq \lambda\norm{f_0}_{H^s}\,, 
\end{equation}
so that almost surely, 
\begin{equation*}
\norm{\japbrak{\nabla}^s\e^{it\Delta}f_0^\omega}_{L_t^q(\R; L_x^{\widetilde{r}}(\R^3))} < \infty\,.
\end{equation*}
\end{lem}

\subsection{Probabilistic local well-posedness}\label{sec:proba-lwp}

With the above refined probabilistic Strichartz estimates at hand, one can observe some nonlinear smoothing effect. Specifically, starting from an initial data $f_0^\omega\in H^s$ that satisfies~\eqref{eq:prob-str}, and  by performing a \emph{linear-nonlinear} decomposition through the ansatz
\[
u(t) = \e^{it\Delta}f_0^\omega + v\,, 
\]
we run a contraction mapping argument for $v$ in $H^\sigma$ with $\frac{1}{2}<\sigma<2s$, where $v$ is solution to a cubic Schrödinger equation with stochastic forcing terms depending on $f(t)=e^{it\Delta}f_0^{\omega}$, and zero initial condition: 
\begin{equation}
\label{eq:NLS-f}
\tag{NLS$_f$}
    \begin{cases}
    (i\partial_t+\Delta)v = \abs{v+f}^2(v+f)\,,\\
    v(0) = 0\,.
    \end{cases}
\end{equation}
We use the following \textit{refined Strichartz} norm for the linear evolution of the initial data 

\[
\norm{f}_{\widetilde{S}^s} := \norm{\japbrak{\nabla}^sf}_{L_{t,x}^4(\R\times\R^3)} + \norm{\japbrak{\nabla}^sf}_{L_{t,x}^5(\R\times\R^3)}+\norm{\japbrak{\nabla}^sf}_{L_{t,x}^\frac{30}{7}(\R\times\R^3)}\,, 
\]
It follows from Lemma~\ref{lem:prob-str} that for all $\lambda>0$, we have $\lambda$-certainly that
\[
\norm{\e^{it\Delta}f_0^\omega}_{\widetilde{S}^s} \leq \lambda\norm{f_0}_{H^s}\,,
\]
 As a consequence, the set
\[
\Sigma := \set{ f_0^\omega \mid\ \norm{\e^{it\Delta}f_0^\omega}_{\widetilde{S}^s} <\infty}\,
\]
has full $\mu$-measure. To control the source term in~\eqref{eq:NLS-f} we set
\[
Y^s := L_t^\infty(\R; H^s)\cap \widetilde{S}^s\,.
\]
 The local existence result due to Bényi, Oh and Pocovnicu for~\eqref{eq:NLS} with initial condition $f_0^\omega\in\Sigma$, that we now recall, states as follows.
\begin{thm}[Almost sure local well-posedness, see~\cite{benyi2015-local,benyi2015}]
\label{theo-as-lwp}
For all $f_0^\omega\in\Sigma$, there exists $T=T(\norm{\e^{it\Delta}f_0^\omega}_{Y^s})>0$ such that the Cauchy problem~\eqref{eq:NLS} with initial data $f_0^{\omega}$ has a unique solution 
\[
u\in \e^{it\Delta}f_0^\omega+X_T^{\sigma,b}\hookrightarrow L_t^\infty(\intervalcc{0}{T};H^s)\,,
\]
for some $b=\frac{1}{2}^+$. Moreover, $u\in C_t(\intervalcc{0}{T};H^s)$\,.
\end{thm}

Before sketching the proof, let us comment on the \emph{Bourgain recentering method} used to obtain such a result. The idea is to reduce the Cauchy problem~\eqref{eq:NLS} with initial data $f_0^\omega$ to the Cauchy problem~\eqref{eq:NLS-f}, with a forcing term $f=\e^{it\Delta}f_0^\omega$ and with zero initial condition, at scaling-subcritical regularity $H^\sigma(\R^3)$. This is the reason why we have uniqueness for $u$ in an affine subspace of $L^\infty_t H^s_x(\R^3)$. This strategy generates strong solutions for initial data in $\Sigma$, in the sense that the flow map is continuous from the time interval of existence to a functional space embedded in $H^s$. Then, the contraction mapping argument follows from trilinear estimates in Fourier restriction spaces, required to handle the bilinear estimates from Lemma~\ref{lem:bili}. They are associated to the norm 
\[
\|f\|_{X^{\sigma,b}}=\|\langle \xi\rangle^{\sigma}\langle \tau+|\xi|^2\rangle^b\widehat{f}(\tau,\xi)\|_{L^2_{\tau,\xi}(\R\times\R^3)}\,,\ \|f\|_{X^{\sigma,b}_T}=\inf\left\{ \|F\|_{X^{\sigma,b}}  \mid F\in X^{\sigma,b}\,,\ F|_{[0,T]}=f\right\}\,.
\]

Combining the bilinear estimate from Lemma~\ref{lem:bili} with the probabilistic Strichartz estimate from lemma~\ref{lem:prob-str}, we obtain trilinear estimates based on a Littlewood-Paley analysis. Before proceeding further, let us simplify the notation. We write the cubic nonlinearity as a trilinear operator, acting on functions $\set{u_i}_{1\leq i\leq3}$ of type either $\e^{it\Delta}f_0^\omega$ or $v$. In this purpose, we define
\[
\mathcal{N}(u_1,u_2,u_3) := u_1\overline{u_2}u_3\,,\quad \mathcal{N}(u) := \mathcal{N}(u,u,u) = \abs{u}^2u\,.
\]
Once again, we drop the complex conjugation sign since it plays no role in our short time perturbative analysis, so that the above notation is invariant under permutation of the terms $u_1,u_2,u_3$. For convenience, we write the Duhamel integral term as
$
(i\partial_t-\Delta)^{-1} (F) : = \int_0^t \e^{i(t-\tau)\Delta}F(\tau)d\tau$.
\begin{lem}[Trilinear estimates, see Proposition 3.1 in~\cite{benyi2015-local,benyi2015}]
\label{lem:trilinear}
Let $\frac{1}{4}<s<\frac{1}{2}$ and $\frac{1}{2}<\sigma<2s$. There exist $C>0$, $b=\frac{1}{2}^+$ and $\theta=0^+$

 such that for all $0<T\lesssim1$, for all $f_1,f_2,f_3\in Y^s$ and $v, v_1,v_2,v_3\in X^{\sigma,b}(\intervalcc{0}{T})$, we have 

\begin{align*}
    \norm{\parent{i\partial_t-\Delta}^{-1}\mathcal{N}(v_1,v_2,v_3)}_{X_T^{\sigma,b}} &\leq CT^\theta\norm{v_1}_{X_T^{\sigma,b}}\norm{v_2}_{X_T^{\sigma,b}}\norm{v_3}_{X_T^{\sigma,b}}\,,\\
    \norm{\parent{i\partial_t-\Delta}^{-1}\mathcal{N}(f_1,f_2,f_3)}_{X_T^{\sigma,b}} &\leq CT^\theta\norm{f_1}_{Y^s}\norm{f_2}_{Y^s}\norm{f_3}_{Y^s}\,,\\
    \norm{\parent{i\partial_t-\Delta}^{-1}\mathcal{N}(v_1,v_2,f_3)}_{X_T^{\sigma,b}} &\leq CT^\theta\norm{v_1}_{X_T^{\sigma,b}}\norm{v_2}_{X_T^{\sigma,b}}\norm{f_3}_{Y^s}\,,\\
    \norm{\parent{i\partial_t-\Delta}^{-1}\mathcal{N}(v_1,f_2,f_3)}_{X_T^{\sigma,b}} &\leq CT^\theta\norm{v_1}_{X_T^{\sigma,b}}\norm{f_2}_{Y^s}\norm{f_3}_{Y^s}\,.
\end{align*}
In addition, we have the Lipschitz estimate 
\[
\norm{\parent{i\partial_t-\Delta}^{-1}\parent{\mathcal{N}\parent{v_1+f}-\mathcal{N}\parent{v_2+f}}}_{X^{\sigma,b}_T} \leq CT^\theta\norm{v_1-v_2}_{X^{\sigma,b}_T}\parent{\norm{v}_{X^{\sigma,b}_T}^2+\norm{f}_Y^2}\,.
\]
\end{lem}

As a consequence of the above trilinear estimates, we have a good Cauchy theory for~\eqref{eq:NLS-f}.
\begin{lem}[Local well-posedness for the perturbed equation]
\label{lem:eq-NLSf} Let $R>0$ and assume that 
\[
\norm{f}_{Y^s}\leq R\,.
\]
For any $\frac{1}{2}<\sigma<2s$ and $b=\frac{1}{2}^+$, there exist $T=T(R)>0$ and a unique solution $v$ to~\eqref{eq:NLS-f} in $X_T^{\sigma,b}$. Moreover, 
\[
v\in C(\intervalcc{0}{T};H^\sigma)\,,\quad \text{and}\ \norm{v}_{X_T^{\sigma,b}} \leq R\,.
\]
\end{lem}
Theorem~\ref{theo-as-lwp} follows from Lemma~\ref{lem:eq-NLSf} applied with $f:=\e^{it\Delta}f_0^\omega$, for some initial data $f_0^\omega\in\Sigma$.
\begin{proof}[Proof of Lemma~\ref{lem:eq-NLSf}]
We show that for $T$ small enough with respect to $R$, the map
\[
\Gamma : v\in X_T^{\sigma,b} \longmapsto (i\partial_t-\Delta)^{-1}\parent{\abs{v+f}^2(v+f)}\in X_T^{\sigma,b}
\]
is a contraction on the ball
\[
B_R = \set{v\in X_T^{\sigma,b}\mid \norm{v}_{X_T^{\sigma,b}}\leq R}\,.
\]
It follows from the trilinear estimates of Lemma~\ref{lem:trilinear} that for all $v,v_1,v_2\in X_T^{\sigma,b}$, 
\begin{align*}
    \norm{\Gamma v}_{X_T^{\sigma,b}} &\leq CT^\theta\parent{\norm{v}_{X^{\sigma,b}}^3+\norm{f}_{Y^s}^3}\,,\\
    \norm{\Gamma(v_2)-\Gamma(v_1)}_{X_T^{\sigma,b}}&\leq CT^\theta \norm{v_2-v_1}_{X_T^{\sigma,b}}\parent{\norm{v_1}_{X_T^{\sigma,b}}^2+\norm{v_2}_{X_T^{\sigma,b}}^2+\norm{f}_{Y^s}^2}\,,
    \end{align*}
so that 
\begin{align*}
   \norm{\Gamma v}_{X_T^{\sigma,b}} &\leq 2CT^\theta R^3\,,\\
   \norm{\Gamma(v_2)-\Gamma(v_1)}_{X_T^{\sigma,b}}&\leq 2CT^\theta R^2 \norm{v_2-v_1}_{X_T^{\sigma,b}}\,.
\end{align*}
We conclude by choosing $T=T(R)>0$ such that $2CT^\theta R^2\leq\frac{1}{2}$.
\end{proof}
\subsection{Approximate solutions by convolution and convergence}
\label{sec:proba-convol}

Let $f_0^{\omega}\in\Sigma$, and $u= \e^{it\Delta}f_0^{\omega} + v$ be the local solution to~\eqref{eq:NLS} obtained in Theorem~\ref{theo-as-lwp}, with $
v \in X^{\sigma,b}(\intervalcc{0}{T}) \subset C(\intervalcc{0}{T}, H^\sigma)$. Let $(\epsilon_k)_{k\geq0}$ be a sequence of positive numbers that go to zero. Using the approximate identity~\eqref{eq:rho-eps}, we define a sequence of regularized initial data
\[
f_{0,k}^{\omega} := f_0^{\omega}\ast\rho_{\epsilon_k}\in H^\infty\,.
\]
Then, we denote $u_k$ the maximal solution in $H^1$ to~\eqref{eq:NLS}. Using the conservation of the coercive energy (in the defocusing case), $u_k$ is global. The purpose of this section is to prove the convergence claimed in Theorem~\ref{theo-lwp}. We set $R>0$ be such that 
\[
\norm{\e^{it\Delta}f_0^\omega}_{Y^s}\leq R
\,,\]
and $T=T(R)$ is as in the proof of Lemma~\ref{lem:eq-NLSf}, applied with $f:= \e^{it\Delta}f_0^{\omega}$.
\begin{prop}
We have the convergence 
\[
\underset{k\to\infty}{\lim} \norm{u_k-u}_{L^\infty(\intervalcc{0}{T},H^s)} = 0\,.
\]
\end{prop}

\begin{proof}
We do not have a good Cauchy theory in $H^s(\R^3)$. However, Theorem~\ref{theo-as-lwp} provides a Cauchy theory for the perturbed equation~\eqref{eq:NLS-f}, with a forcing term $f=e^{it\Delta}f_0^{\omega}$ in $Y^s$. Hence, it is relevant to run a linear-nonlinear decomposition of $u_k$ as well
\[
u_k = \e^{it\Delta}f_{0,k}^{\omega} + v_k =: f_k + v_k\,, \quad f_k=e^{it\Delta}f_{0,k}^{\omega}\,,
\]
We define
\[
\delta_k := f_k-f= \e^{it\Delta}(f_0^{\omega}\ast \rho_{\epsilon_k}-f_0^{\omega})\,,
\]
and note that the convergence
\[
\underset{k\to\infty}{\lim}\norm{\delta_k}_{L_t^\infty H^s_x} = \underset{k\to\infty}{\lim}\norm{f_0-f_0\ast\rho_{\epsilon_k}}_{H^s_x} =0\,.
\]
follows from the unitary property of the linear evolution together with the fact that the convolution commutes with the linear evolution. Then, we use the embedding $X^{\sigma,b}(\intervalcc{0}{T})\hookrightarrow L^\infty(\intervalcc{0}{T},H^\sigma)$ to reduce the proof of Theorem~\ref{theo-lwp} to the convergence of the nonlinear term $v_k$ in the Fourier restriction space. 
\begin{equation}
    \label{eq:conv-v}
\underset{k\to\infty}{\lim} \norm{v-v_k}_{X^{\sigma,b}_T} = 0\,.
\end{equation}
To achieve such a goal, we start by proving that 
\begin{equation}\label{eq:deltak}
\underset{k\to\infty}{\lim}\norm{\delta_k}_{\widetilde{S}^s}=0\,.
\end{equation}
We fix $2\leq q,r<\infty$ and $\omega\in\Omega$ such that 
\[
\norm{e^{it\Delta}f_0^{\omega}}_{L_t^qL_x^r}<\infty\,.
\]
In particular, we observe that
\begin{equation}
\label{eq:leb-t}
    \norm{e^{it\Delta}f_0^{\omega}}_{L_x^r}<\infty\,,\quad \text{\emph{Lebesgue a.e.}}\ t\in\R\,.
\end{equation}
In addition, when $t$ is such that~\eqref{eq:leb-t} holds we
\[
\norm{\delta_k(t)}_{L_x^r} \leq (1+\|\rho\|_{L^1_x})\norm{\e^{it\Delta}f_0^\omega}_{L_x^r}\,,
\]
so that the sequence of measurable functions $t\to\norm{\delta_k(t)}_{L_x^r}$ is uniformly bounded in $k$. From the assumption on $f_0^\omega$, we also have the uniform bound
\begin{equation}
    \label{eq:cvd}
    \norm{\delta_k}_{L_t^qL_x^r} \leq (1+\|\rho\|_{L^1_x})\norm{\e^{it\Delta}f_0^\omega}_{L_t^qL_x^r}\,.
\end{equation}
Next, we use the fact that the approximate identity converge in $L_x^r$, and commutes with the linear evolution $\e^{it\Delta}$, to obtain
\[
\underset{k\to\infty}{\lim}\|\delta_k(t)\|_{L^r_x}
	=\underset{k\to\infty}{\lim}\|e^{it\Delta}\parent{f_0^{\omega}-f_0^{\omega}\ast\rho_{\epsilon_k}}\|_{L^r_x} = \underset{k\to\infty}{\lim}\|e^{it\Delta}f_0^{\omega}-\parent{e^{it\Delta}f_0^{\omega}}\ast\rho_{\epsilon_k}\|_{L^r_x} = 0\,.
\]
This convergence, combined with the uniform bound~\eqref{eq:cvd}, allows us to apply the dominated convergence theorem in $L_t^q$ and to conclude that the convergence~\eqref{eq:deltak} holds.~\footnote{Note that the only features of this specific approximate identity by convolution we used is that it commutes with $\e^{it\Delta}$, and that it is uniformly bounded in $L^p$-spaces.} Next, we fix $\eta>0$ arbitrarily small, and $k_0\in\N$ such that for all $k\geq k_0\,,\ 
\norm{\delta_k}_{Y^s}\leq\eta$. We define
\[
w_k := v_k-v\in X^{\sigma,b}(\intervalcc{0}{T})\,.
\]
By using the equations for $v$ and $v_{k}$, we see that $w_{k}$ is solution to the following forced Schrödinger equation with zero initial condition
\begin{equation*}
\begin{cases}
    (i\partial_t+\Delta)w_k = \mathcal{N}(f_{k}+v_{k}) - \mathcal{N}(f+v)= \mathcal{N}(w_k+\delta_{k}+f+v) - \mathcal{N}(f+v) =: F\,.\\
    w_k(0) = 0\,,
\end{cases}
\end{equation*}
We perform a bootstrap argument and propagate the smallness of the $X^{\sigma,b}$-norm of $w_k$ to prove that for all $k\geq k_0$,
\[
\norm{w_{k}}_{X^{\sigma,b}(\intervalcc{0}{T})} \leq 2\eta\,.
\]
Up to some multiplicative universal constants, the forcing term writes 
\[
F \sim \mathcal{N}(w_{k}+\delta_{k}) + \mathcal{N}(w_{k}+\delta_{k},w_{k}+\delta_{k},f+v)+ \mathcal{N}(w_{k}+\delta_{k},  f+v , f+ v)\,.
\]
Using the trilinear estimates from Lemma~\ref{lem:trilinear}, we get the following a priori estimate
\begin{multline*}
\norm{w_{k}}_{X_T^{\sigma,b}} \leq C T^\theta \norm{w_{k}}_{X_T^{\sigma,b}}\parent{\norm{w_{k}}_{X_T^{\sigma,b}}^2+\norm{\delta_{k}}_{Y^s}^2+\norm{v}_{X_T^{\sigma,b}}^2+\norm{f}_{Y^s}^2} \\
+ CT^\theta\norm{\delta_{k}}_{Y^s}\parent{\norm{w_{k}}_{X_T^{\sigma,b}}^2+\norm{\delta_{k}}_{Y^s}^2+\norm{v}_{X_T^{\sigma,b}}^2+\norm{f}_{Y^s}^2}\,.
\end{multline*}
We finally exploit the smallness of the $Y^s$-norm of $\delta_{k}$ when $k\geq k_0$, the estimates on $\norm{v}_{X_T^{\sigma,b}}$ from Theorem~\ref{theo-as-lwp} and the a priori estimate on $\norm{w_k}_{X_T^{\sigma,b}}$ from the bootstrap hypothesis to obtain
\[
\norm{w_k}_{X_T^{\sigma,b}} \leq 2CT^\theta\norm{w_k}_{X_T^{\sigma,b}}\parent{\eta^2+R^2} + 2CT^\theta\eta\parent{\eta^2+R^2}\,.
\]
Since we chose $\eta$ arbitrarily small, and $T$ is chosen such that $2CT^\theta (1+R^2)\leq \frac{1}{2}$ we conclude that,
\[
\norm{w_k}_{X_T^{\sigma,b}} \leq \frac{1}{2}\norm{w_k}_{X_T^{\sigma,b}} + \frac{1}{2}\eta\,.
\]
This finishes the proof of the bootstrap.
\end{proof}

\bibliography{mybib.bib}{}
\bibliographystyle{alpha}

\end{document}